\documentclass[11pt,oneside,leqno]{amsart}
\usepackage[utf8]{inputenc}
\usepackage[english]{babel}
\usepackage{amssymb,amsmath,latexsym,amsthm,amsfonts,bbm,dsfont}
\usepackage{geometry}
\usepackage{float}
\usepackage[colorlinks=true,citecolor=red,linkcolor=blue,pdfpagetransition=Blinds]{hyperref}
\numberwithin{equation}{section}
\usepackage{mathtools, stmaryrd}
\usepackage{natbib}
\usepackage{textgreek}
\usepackage{xcolor}
\numberwithin{equation}{section}
\usepackage{xcolor}
\usepackage{graphicx}

\newcommand{\E}{\mathbb{E}}

\newtheorem{theorem}{Theorem}[section]

\newtheorem{lemma}[theorem]{Lemma}

\newtheorem{remark}[theorem]{Remark}

\begin{document}

\title{Controllability for semi-discrete semilinear stochastic parabolic operators }

\author[R. Lecaros]{Rodrigo Lecaros}
\address[R. Lecaros]{Departamento de Matem\'atica, Universidad T\'ecnica Federico Santa Mar\'ia,  Santiago, Chile.}
\email{rodrigo.lecaros@usm.cl}
\author[A. A. P\'erez]{Ariel A. P\'erez}
\address[A. A. P\'erez]{Departamento de Matem\'atica, Universidad del B\'io-B\'io, Concepci\'on, Chile.}
\email{aaperez@ubiobio.cl}
\author[M. F. Prado]{Manuel F. Prado}
\address[M. F. Prado]{(Corresponding Author)Departamento de Matem\'atica, Universidad T\'ecnica Federico Santa Mar\'ia,  Santiago, Chile.}
\email{mprado@usm.cl}


\begin{abstract}
In \cite{LPP:2025}, it was shown that, in arbitrary dimension, the spatial semi-discretization of a controlled stochastic parabolic operator is generically not null-controllable. Nevertheless, $\phi$-null controllability results remain attainable. The present paper extends those results to semi-discrete semilinear stochastic operators in arbitrary dimension, whose nonlinearities may also depend on the first-order spatial derivatives. The approach relies on establishing a new Carleman estimate for the adjoint backward stochastic parabolic operator, which yields $\phi$-null controllability for the associated linear system via a duality argument. The semilinear case is handeld by means of a fixed-point argument. As particular cases, our results recover the one-dimensional linear results of \cite{zhao:2024}, the multidimensional linear results of \cite{LPP:2025}, and the semilinear one-dimensional framework of \cite{WZ:2025} in the absence of gradient dependence.
\end{abstract}

\maketitle

\section{Introduction}
In the deterministic setting, the effect of spatial discretization on controllability has been studied extensively. As established in \cite{zuazua2005propagation}, discretization and controllability do not, in general, commute: even when the continuous system is null-controllable, its semi-discrete approximation may fail to retain this property. To address this obstruction, the notion of $\phi$-null controllability was considered \cite{BHLR:2010a,BHLR:2010b,BLR:2014,LT:2006}. This weaker notion consists in constructing uniformly bounded controls such that the norm of the discrete solution at a fixed time $T$ decays at a prescribed rate $\phi(h)$, where $h>0$ denotes the spatial mesh size and $\phi(h)\longrightarrow 0$ as $h\longrightarrow 0$. This framework has been developed in a broad range of contexts, including semi-discrete spatial approximations  \cite{AB:2020,ABM:2018,BHSDT:2019,CLTP-2022,Thuy:2014,Thuy:2015}, fully discrete schemes \cite{BHLR:2011,LMPZ-2023,GC-HS-2021,AA:perez:2024}, and time-discrete settings \cite{HS2023,BDK:2026}. General expositions of this controllability notion can be found in \cite{B:2013,Thuy:2015,TZ-2009}.

Analogous difficulties arise in the stochastic setting. In \cite{LPP:2025}, it is shown that spatial semi-discretizations of controlled stochastic parabolic equations are, in general, not null-controllable in dimension $n\geq 2$. This negative result motivates the study of $\phi$-null controllability for semi-discrete stochastic systems. Several contributions in this direction have been obtained in the one-dimensional case: the linear setting is addressed in \cite{zhao:2024,LPP:2025}, the semilinear setting without gradient dependence in \cite{WZ:2025}, and the fourth-order linear case in \cite{WZ:2024}. All these works establish $\phi$-null controllability results but leave open the question of whether the corresponding semi-discrete systems actually fail to be null-controllable. We note, by contrast, that in the deterministic one-dimensional setting with constant coefficients, null-controllability does hold at the discrete level; see \cite{lopez1998some}. The analogous question for variable coefficients remains open.

 The objective of the present paper is to generalize the semilinear one-dimensional results of \cite{WZ:2025} to arbitrary spatial dimensions, and to extend the linear multidimensional results of \cite{LPP:2025} to a semilinear framework. Recall that \cite{LPP:2025} itself extended the one-dimensional linear results of \cite{zhao:2024} to arbitrary spatial dimensions, under weaker assumptions on the diffusion coefficient and for a broader class of semi-discrete operators. The present work builds upon the methodology of \cite{LPP:2025} and its main novelty is the proof of $\phi$-null controllability for semi-discrete semilinear stochastic parabolic operators in arbitrary spatial dimension, under globally Lipschitz nonlinearities depending on both the state and its discrete spatial gradient. The approach proceeds in two steps: a variational argument combined with a new semi-discrete Carleman estimate for the adjoint backward stochastic parabolic operator yields $\phi$-null controllability for the associated linear system; the semilinear case is then handled via a fixed-point argument. As a consequence, our results recover, as particular cases, the one-dimensional linear results of \cite{zhao:2024}, the multidimensional linear results of \cite{LPP:2025}, and the semilinear one-dimensional framework of \cite{WZ:2025} in the absence of gradient dependence.
\subsection{Notation and assumptions}
Let $(\Omega, \mathcal{F}, \{\mathcal{F}_{t}\}_{t \geq 0}, \mathds{P})$ be a complete filtered probability space on which a one-dimensional standard Brownian motion $\{B(t)\}_{t \geq 0}$ is defined. We assume that $\{\mathcal{F}_{t}\}_{t \geq 0}$ is the natural filtration generated by $B(\cdot)$, augmented by all $\mathds{P}$-null sets in $\mathcal{F}$, and we denote by $\mathds{F}$ the progressive $\sigma$-field with respect to $\{\mathcal{F}_{t}\}_{t \geq 0}$.

Let $H$ be a Banach space. We denote by $C([0,T];H)$ the Banach space of all strongly continuous $H$-valued functions on $[0,T]$. We further introduce the following function spaces:      $L^2_{\mathcal{F}_t}(\Omega;H)$ denotes the space of all $\mathcal{F}_t$-measurable random variables $\zeta$ with $\mathbb{E}|\zeta|_{H}^2 < \infty$; $L^2_{\mathds{F}}(0,T;H)$ denotes the Banach space consisting of all $H-$valued $\mathds{F}$-adapted processes $X(\cdot)$ such that $\E(|X|_{L^2(0,T;H)}|)<\infty$, endowed with the canonical norm; $L^\infty_{\mathds{F}}(0,T;H)$ denotes the Banach space consisting of all $H$-valued $\mathds{F}$-adapted essentially bounded processes; and $L^2_{\mathds{F}}(\Omega;C(0,T;H))$ denotes the Banach space of all $H$-valued $\mathds{F}$-adapted continuous processes $X$ satisfying $\mathbb{E}(|X|^2_{C([0,T];H)}) < \infty$, endowed with the canonical norm. More generally, one defines $L_{\mathds{F}}^2(\Omega;C^{m}([0,T];H))$ analogously for any positive integer $m$.

Let $n \geq 2$, $N \in \mathbb{N}$, and $T > 0$ be fixed. Consider the domain 
$G := (0,1)^n$, and let $G_0\subset G$ be a non-empty open subset. The mesh size is defined by $h := 1/(N+1)$. The one-dimensional  grid on $(0,1)$ is then given by $\mathcal{K} :=\{ x_i = i h \mid i = 1,\dots,N \}$, and the regular partition of $G$ is $\mathcal{M} := G \cap \mathcal{K}^n$, with $\mathcal{M}_{0}:=G_{0}\cap\mathcal{K}^{n}$. Now, we define the dual mesh in the direction $e_{i}$ by 
\begin{equation*}
    \mathcal{M}_{i}^{\ast} := \left\{ x + \frac{h}{2} e_{i} \mid x \in \mathcal{M} \right\}\cup\left\{ x - \frac{h}{2} e_{i} \mid x \in \mathcal{M} \right\},
\end{equation*}where $\{e_{i}\}_{i=1}^{n}$ denotes the canonical basis of $\mathbb{R}^{n}$. The mesh obtained by applying the dual operation successively in directions $e_{i}$ and $e_{j}$ is denoted by $\overline{\mathcal{M}}_{ij} := (\mathcal{M}_{i}^{\ast})_{j}^{\ast}$. 
In addition, we define the boundary of the set $\mathcal{M}$ in direction $e_{i}$ by $\partial_{i}\mathcal{M}:=\overline{\mathcal{M}_{ii}}\setminus\mathcal{M}$. Thus, the boundary and closure of a set $\mathcal{M}$ is given by $\displaystyle
       \partial\mathcal{M} := \bigcup_{i=1}^{n} \partial_{i}\mathcal{M}$ and $\overline{\mathcal{M}} := \mathcal{M} \cup \partial \mathcal{M}$.
We denote by $C(\mathcal{M})$ the set of real-valued functions defined on the mesh $\mathcal{M}$. We define the average and the difference operators as the operators from $C(\overline{\mathcal{M}})$ to $C(\mathcal{M}_{i}^{\ast})$:
\begin{equation*}
\begin{split}
    A_{i}u(x) &:= \frac{1}{2}(u(x+\frac{h}{2} e_{i}) + u(x-\frac{h}{2} e_{i}) ); \,
    D_{i}u(x) := \frac{1}{h}( u(x+\frac{h}{2} e_{i}) - u(x-\frac{h}{2} e_{i})).
\end{split}
\end{equation*} 
Then, for a fixed $h$, we define $L^{2}_{h}(\mathcal{M})$-norm by
$\displaystyle \|u\|^{2}_{L^{2}_{h}(\mathcal{M})}:=h^{n}\sum_{x\in\mathcal{M}}|u(x)|^{2}$. Similarly, we define the norm $H^{1}(\mathcal{M})$-norm by $\displaystyle \|u\|^{2}_{H^{1}(\mathcal{M})}=\|u\|^{2}_{L^{2}_{h}(\mathcal{M})}+\sum_{i=1}^{n}\|D_{i}u\|^{2}_{L^{2}_{h}(\mathcal{M}^{\ast}_{i})}$. Here and throughout, $C$ denotes a generic constant, which may change from line to line, but independent of $h$.\\
In this work, using the previous notation, we consider a semi-discrete semilinear stochastic parabolic system given by
 \begin{equation}\label{eq:systemnonlinear}
    \left\{\begin{aligned}
\mathcal{P}y=&(F_1(\omega,t,x,y,\nabla_{h}y)+\mathbbm{1}_{\mathcal{M}_0}u)\,dt+(F_{2}(\omega,t,x,y,\nabla_{h}y)+U)\,dB(t)\,\text{in} \,Q, \\
                y=&0\quad \text{on}\quad\partial Q,\quad \left.y\right|_{t=0}=y_{0}\quad\text{in} \quad\mathcal{M},
    \end{aligned}\right.
    \end{equation}
where  $\mathcal{P}y:= dy-\sum_{i=1}^{n}D_i(\gamma_iD_iy)\,dt$, $(\nabla_{h}y)_i:=A_{i}D_{i}y$ is the $i$-th component of the discrete gradient $\nabla_{h}y\in\mathbb{R}^{n}$, $Q:=(0,T)\times\mathcal{M}$, $\partial Q:=(0,T)\times\partial\mathcal{M}$. System \eqref{eq:systemnonlinear} corresponds to a spatial semi-discretization of the system (1.11) in \cite{ZXL:2025}, where the authors extend the existence of null-controllability results to a more general class of nonlinearity in the continuous setting. \\
   The hypotheses considered throughout this work are the following:
\begin{itemize}
    \item[(A1)] For each $i=1,\ldots, n$, each coefficient  $\gamma_i$ is a positive time-independent function satisfying the following condition: There exists a constant $\gamma_{0}>0$ such that
    $$\text{reg}(\gamma):=
 \operatorname*{ess\,sup}_{\substack{x \in G \\ i = 1,\dots,n}} \left( \gamma_i + \frac{1}{\gamma_i} +|\nabla_{x} \gamma_i|^2 \right) \leq \gamma_{0}.$$
\item[(A2)] The nonlinearities $F_{1}$ and $F_{2}$ satisfy the following conditions:
\begin{itemize}
    \item For each $y\in H_{0}^1(\mathcal{M})$, the processes $F_i(\cdot,\cdot,\cdot,y,\nabla_{h}y)$, $i=1,2$, are $\mathbb{F}-$adapted and $L_{h}^{2}$-valued stochastic processes.
    \item For all $(\omega,t,x)\in \Omega\times (0,T)\times\mathcal{M}$, $F_i(\omega,t,x,0,0)=0$ for $i=1,2$.
    \item There exist constants $L_{i}>0$, $i=1,2$, such that
    \begin{equation*}
        |F_i(\omega,t,x,a_1,b_{1})-F_i(\omega,t,x,a_{2},b_{2})|\leq L_{i}(|a_{1}-a_{2}|+|b_{1}-b_{2}|),\quad i=1,2,
    \end{equation*}
    for all $(\omega,t,x)\in\Omega\times(0,T)\times\mathcal{M}$ and $ (a_{1},b_{1}),(a_{2},b_{2})\in \mathbb{R}\times\mathbb{R}^n$.
\end{itemize}
\end{itemize}

Since, as is shown in \cite{LPP:2025}, null-controllability fails for the linear spatial semi-discretization of a stochastic parabolic equation, we pursue the notion of $\phi$-null controllability for the system \eqref{system:nonlinear}. This consists in constructing a pair of controls $({u},{U})$, uniformly bounded in $h$, such that the norm of the solution at time $T$ is bounded by a function $\phi$ that tends to zero when $h\to 0$. More precisely, we seek controls such that 
\begin{equation*}
    \E\int_{\mathcal{M}}|y(T)|^{2}\leq C\phi(h)\E\int_{\mathcal{M}}|y_{0}|^{2}.
\end{equation*}
A key novelty of the present work, is the incorporation of the discrete gradient $\nabla_{h}y$ as an argument of the nonlinearities $F_{1}$ and $F_{2}$. 
\subsection{Main results}\label{sec:main} 
The primary objective of this work is to analyze the $\phi$-null controllability of semi-discrete semilinear forward parabolic SPDEs  \eqref{eq:systemnonlinear}. To this end, we first introduce the weight functions according to \cite{fursikov-1996}. For a nonempty subset $\mathcal{M}'$ of $G$ such that $\overline{\mathcal{M}'}\subset\mathcal{M}_{0}$, there exists a function $\psi\in C^4(\overline{G};[0,1])$ such that
\begin{equation}\label{assumtion:psi}
 0< \psi(x)\leq 1 \,\mbox{in}\, G,\quad \psi(x)=0\,\mbox{on}\;\,\partial G\quad \mbox{and}\quad\inf_{x\in G\setminus \overline{G}_1}|\nabla \psi(x)|\geq \alpha>0. 
\end{equation}    
For $\lambda>1$ and $m\geq 1$, we define the function
\begin{align}\label{funcion-peso-2}
    \varphi(x)&=\xi(x)-\lambda e^{6\lambda(m+1) },
\end{align}
with $\xi(x)=e^{\lambda(\psi(x)+6m)}$ and for $0<\delta < 1/2$, we define $\theta\in C^2([0,T])$ by
\begin{equation}\label{eq:temporalWeightfunctional}
    \theta(t)=\begin{cases}
        1+\left(1-4T^{-1}t\right)^{\sigma} &\,t\in [0,T/4],\\
        1&t\in [T/4,T/2],\\
        \mbox{is increasing}& t\in [T/2,3T/4],\\
        [T-t+\delta T]^{-m} & t\in [3T/4,T].
    \end{cases}
\end{equation}
where $\sigma$ is defined as
\begin{equation*}
    \sigma=\tau\lambda^2e^{\lambda(6m-4)}>2,\quad\mbox{for all}\, \lambda\geq 1.
\end{equation*}
Given $\tau\geq 1$, we set $s(t)=\tau\theta(t)$, $r := e^{s\varphi}$ and $\rho = r^{-1}$.  
\begin{remark}
    The main difference between \eqref{eq:temporalWeightfunctional} and  continuous setting \cite{ZXL:2025,HS:LB:P-2023} is that we avoid the blow up in $t=T$ by using the parameter $\delta\in (0,1/2)$, because $\theta$ must be bounded to obtain asymptotic behavior of the weight functions $r$ and $\rho$ when the discrete operators are applied.  
\end{remark}
Our main result, known as $\phi$-null controllability, holds for any function $\phi:(0,\infty)\rightarrow (0,\infty)$ satisfying $\lim_{h\rightarrow 0}\phi(h)=0$ and 
\begin{equation}\label{condphi}\liminf_{h\rightarrow 0}\phi(h)/e^{-\kappa h^{-1}}>0.
\end{equation}
\begin{theorem}\label{theo:mainresult} Suppose that assumptions $(A_{1})$ and $(A_{2})$ hold. Then there exist constants $\kappa>0$, $C_{0}>0$, and $h_{0}>0$ depending on $G_{0}$ , $T$, $\gamma_{0}$ $L_{1}$, $L_{2}$ but independent of $h$ and $y_{0}$, such that for every $h\in(0,h_{0})$ and every initial condition $y_{0}\in L^{2}_{\mathcal{F}_{0}}(\Omega;L^{2}_{h}(\mathcal{M}))$, there exists a control pair 
\begin{equation*}
    (u,U)\in L^{2}_{\mathbb{F}}(0,T;L_{h}^{2}(\mathcal{M}_{0}))\times L_{\mathbb{F}}^{2}(0,T;L_{h}^{2}(\mathcal{M}))
\end{equation*}
such that the solution $y$ of system \eqref{eq:systemnonlinear} satisfies
\begin{equation*}
\begin{split}
        \E\int_{\mathcal{M}}|y(T)|^2\leq  C\phi(h)\E\int_{\mathcal{M}}|y_{0}|^{2},
\end{split}
\end{equation*}
and
\begin{equation*}
\begin{aligned}
\E\int_{0}^T\int_{\mathcal{M}_0}s^{-3}\lambda^{-4}\xi^{-3}\rho^{2}|u|^2\,dt
+&\E\int_{Q}s^{-2}\lambda^{-2}\xi^{-3}\rho^2|U|^2\,dt\\
&\leq C\left(\E\int_{\mathcal{M}}\tau^{-1}\lambda^{-2}e^{-2\lambda(6m+1)}e^{-4\tau\varphi}|y_0|^2\right),
\end{aligned}
\end{equation*}
where $\phi(h):=Ce^{-\kappa /h}$.
\end{theorem}
\begin{remark}
     Notice that in particular, considering $n=1$, and $\gamma_{i}=1$ and $\nabla_{h} y=0$ we recover the results presented in \cite{zhao:2024}, and for $F_{2}$ and $F_{1}$ linear functions we recover the results from \cite{WZ:2025}. Moreover, just considering $F_{1}$ and $F_{2}$ as linear functions in \eqref{eq:systemnonlinear} we recover the results presented in \cite{LPP:2025}.   Hence, Theorem \ref{theo:mainresult} stands for a generalization of the results presented in \cite{zhao:2024}, \cite{WZ:2025} and \cite{LPP:2025}.
\end{remark}
Assume that for every $y_{0}\in L^{2}_{\mathcal{F}_{0}}(\Omega;L^{2}_{h}(\mathcal{M}))$ there exist a pair of controls $(u,U^{\ast})\in L^{2}_{\mathbb{F}}(0,T;L_{h}^{2}(\mathcal{M}_{0}))\times L_{\mathbb{F}}^{2}(0,T;L_{h}^{2}(\mathcal{M}))$ such that the following system 
\begin{equation}\label{system:nonlinear}
    \left\{\begin{aligned}
\mathcal{P}y=&(F_1(\omega,t,x,y,\nabla_{h}y)+\mathbbm{1}_{\mathcal{M}_0}u)\,dt+U^{\ast}\,dB(t)\quad \text{in} \, Q, \\
                y=&0\quad \text{on}\quad\partial Q,\quad \left.y\right|_{t=0}=y_{0}\quad\text{in} \quad\mathcal{M},
    \end{aligned}\right.
    \end{equation}
 is $\phi$-null controllable. Defining $U:=U^{\ast}-F_{2}(\omega,t,x,\nabla_{h}y)$, we notice that thanks to assumptiion (A2) we have  $F_{2}(\omega,t,x,y,\nabla_{h}y)\in L^{2}_{\mathcal{F}}(0,T;L^{2}_{h}(\mathcal{M}))$ since the solution $y$ of \eqref{system:nonlinear} verifies $y\in L^{2}_{\mathcal{F}}(0,T;H^{1}_{0}(\mathcal{M}))$, and by assumption we have $U^{\ast}\in L^{2}_{\mathcal{F}}(0,T;L^{2}_{h}(\mathcal{M}))$. Moreover, $y$ solves \eqref{eq:systemnonlinear} with controls $(u,U)$. Consequently, the proof of
Theorem \ref{theo:mainresult} reduces to the case  $F_{2}(\cdot) \equiv 0$. For this reason, our main focus is the $\phi$-null controllability for the system \eqref{system:nonlinear}. \\
To deal with the nonlinearity in \eqref{system:nonlinear} it is necessary to obtain the $\phi$-null controllability result for the following linear forward semi-discrete parabolic system:
\begin{equation}\label{EQ:LFSPE}
    \left\{\begin{aligned}
        \mathcal{P}y&=\left(\sum_{i=1}^{n}a_{1i}A_iD_i(y)+a_2y+v+\mathbbm{1}_{\mathcal{M}_0}u\right)dt+U\,dB(t),\\
        y&=0\quad\text{on}\,\, \partial Q, \quad \left.y\right|_{t=0}=y_{0}\quad \text{in}\,\, \mathcal{M}.
    \end{aligned}\right.
\end{equation}
where $(u,U)\in L^2_{\mathbb{F}}(0,T; L_{h}^2(\mathcal{M}_0)) \times L^2_{\mathbb{F}}(0,T; L_{h}^2(\mathcal{M}))$ is a pair controls, $y$ denote the state variable associated with the initial state $y_0\in L^2_{\mathcal{F}_0}(\Omega;L_{h}^{2}(\mathcal{M}))$ and we assume that $a_{1i}\in L_{\mathds{F}}^{\infty}(0,T;L_{h}^{\infty}(\mathcal{M}))$ for $i=1,...,n$, $a_2\in L^{\infty}_{\mathds{F}}(0,T;L_{h}^{\infty}(\mathcal{M}))$ and $v\in L^2_{\mathds{F}}(0,T;L_{h}^2(\mathcal{M}))$.

Since $F_{1}$ depends on both the state and its gradient, it is necessary to obtain suitable estimates of these quantities to apply the fixed-point argument. Consequently, the results established in \cite{zhao:2024} are not applicable in this setting, even when restricted to the one-dimensional case considered in that work because it does not estimate the gradient term. Building on the ideas in \cite{ZXL:2025}, we formulate system \eqref{EQ:LFSPE} and derive the following result, which addresses the requirements for the fixed-point approach and the controllability property of the newly proposed system:

\begin{theorem}\label{lemma:theonullauxiliar}
Let $T>0$, $a_{1i}\in L^{\infty}_{\mathbb{F}}(0,T;L_{h}^{\infty}(\mathcal{M}))$ for $i=1,\ldots,n$, $a_{2}\in L^{\infty}_{\mathbb{F}}(0,T;L_{h}^{\infty}(\mathcal{M}))$,  $v\in L^{2}_{\mathbb{F}}(0,T;L_{h}^{2}(\mathcal{M}))$ and $y_{0}\in L^{2}_{\mathcal{F}_{0}}(\Omega;L_{h}^{2}(\mathcal{M}))$. Then, there exist $\lambda_0 > 0$ such that for all $\lambda > \lambda_0$, the problem admits constants $\tau_0 > 1$ and $\varepsilon_0 > 0$ (depending on $G_{0}$, $c_{0}$, $T$, and $\lambda$), and a pair of control functions $(u, U) \in L^2_{\mathbb{F}}(0,T; L_{h}^{2}(\mathcal{M}_0)) \times L^2_{\mathbb{F}}(0,T; L_{h}^2(\mathcal{M}))$. Consequently, the corresponding solution $y$ to \eqref{EQ:LFSPE} satisfies the following
\begin{equation}\label{eq:boundyz(T)}
\begin{split}
\E \int_{\mathcal{M}}|y(T)|^2\leq& \mathcal{E}_{\lambda,h}\left(\E\int_{Q}s^{-3}\lambda^{-4}\xi^{-3}\rho^2|v|^2\,dt+\E\int_{\mathcal{M}}\tau^{-1}\lambda^{-2}e^{-2\lambda(6m+1)}e^{-4\tau\varphi}|y_0|^2\right)
\end{split}
\end{equation}
and
\begin{equation}\label{eq:estimateCarlmenz}
\begin{split}
\sum_{i=1}^{n}&\E\int_{Q_i^{\ast}}s^{-2}\lambda^{-2}\xi^{-3}\rho^2|D_iy|^2\,dt+\E\int_{Q}\rho^2|y|^2\,dt+\E\int_{0}^T\int_{\mathcal{M}_0}s^{-3}\lambda^{-4}\xi^{-3}\rho^{2}|u|^2\,dt\\
&+\E\int_{Q}s^{-2}\lambda^{-2}\xi^{-3}\rho^2|U|^2\,dt\leq C\left(\E\int_{Q}s^{-3}\lambda^{-4}\xi^{-3}\rho^2|v|^2\,dt\right.\\
&\left.+\E\int_{\mathcal{M}}\tau^{-1}\lambda^{-2}e^{-2\lambda(6m+1)}e^{-4\tau\varphi}|y_0|^2\right),
\end{split}
\end{equation}
for all $\lambda > \lambda_0$, $\tau>\tau_0$, $0<h<h_0$, $0<\delta<1/2$, $\tau(T\delta)^{-m}h\leq \varepsilon_0$ and with $\mathcal{E}_{\lambda,h}:= Ch^{-2}e^{-2s(T)(\lambda-1)e^{6\lambda(m+1)}}$.
\end{theorem}

\begin{remark}
Observe that for each $h$, we obtain a solution $y$ to system \eqref{EQ:LFSPE} that satisfies \eqref{eq:boundyz(T)} and \eqref{eq:estimateCarlmenz}, respectively. However, the right-hand side of inequality \eqref{eq:boundyz(T)}-\eqref{eq:estimateCarlmenz} does not depend on $h$. Therefore, we can recover the result in the continuous setting, and from \eqref{eq:boundyz(T)} we deduce the null controllability in the continuous case. Moreover, by considering $v= 0$, we observe the equivalence with the one-dimensional result in \cite{zhao:2024} or arbitrary dimension result in \cite{LPP:2025}.
\end{remark}

Now, the proof presented in Section~\ref{sec:proofnullcontrollability} of the Theorem~\ref{lemma:theonullauxiliar} relies on an argument based on the minimization of an appropriate functional and a new Carleman estimate applied to the backward equation associated with \eqref{EQ:LFSPE}. Following the ideas in \cite{LPP:2025}, we first obtain a preliminary Carleman estimate for the operator backward $\mathcal{P}^{\ast}z:=dz+\sum_{i=1}^{n}D_i(\gamma_iD_iz)\,dt$ in Appendix~\ref{firstcarleman} by analyzing the modifications introduced by the new weight function proposed in \cite{HS:LB:P-2023,ZXL:2025}. Finally, inspired by \cite{zhao:2024}, this estimate is refined to establish the Carleman inequality required for the proof of Theorem~\ref{lemma:theonullauxiliar}, as follows:
\begin{theorem}\label{theo:Carleman}
Let $\psi$ satisfy assumption \eqref{assumtion:psi} and $\varphi$ according to \eqref{funcion-peso-2}. For $\lambda\geq \lambda_0>1$ sufficiently large, there exist $C$, $\tau_{0}\geq 1$, $h_{0}>0$, $\varepsilon_0 >0$, depending on $G_{0}$, $G_{1}$, $c_{0}$, $T$, and $\lambda$, such that
\begin{equation}\label{eq:InitialCarleman}
\begin{aligned}
    J(w)+&\left.\E\int_{\mathcal{M}}\tau^2\lambda_{0}^3e^{2\lambda_{0}(6m+1)}e^{4\tau\varphi}|w|^2\right|_{t=0}+\sum_{i=1}^{n}\left.\E\int_{\mathcal{M}_{i}^{\ast}}e^{4\tau \varphi(x)}|D_iw|^2\right|_{t=0}\\
    \leq  C&\left(\E\int_{0}^T\int_{\mathcal{M}_{0}}s^{3}\lambda^{4}_{0}\xi^{3}e^{2s\varphi}\,|w|^{2}\,dt+\E\int_{Q}e^{2s\varphi}\,|f|^2\,dt\right.\\
    &\left.+\E\int_{Q}s^2\lambda_{0}^2\xi^2e^{2s\varphi}\,|g|^2\,dt+\frac{1}{h^2}\left.\E\int_{\mathcal{M}}e^{2s(T)\varphi}\,|w|^2\right|_{t=T}\right),
\end{aligned}
\end{equation}
for all $\tau\geq \tau_0 $, $0<h\leq h_0$, $0<\delta<1/2$, $s(t)h\leq \delta_0$, where\break $\displaystyle J(w):=  \E\int_{Q}s^{3}\lambda_{0}^{4}\xi^{3}e^{2s\varphi}\,|w|^{2}\,dt+ \sum_{i=1}^{n}\E\int_{Q_{i}^{\ast}}s\lambda_{0}^{2}\xi e^{2s\varphi}|D_{i}w|^{2}dt$,  $f,g\in L^2_{\mathds{F}}(0,T;L_{h}^{2}(\mathcal{M}))$ and $w$ satisfy $dw+\sum_{i=1}^{n} D_i(\gamma_i D_{i}w)\,dt=fdt+gdB(t)$ with $w=0$ on $\partial \mathcal{M}$.
\end{theorem} 
\begin{remark}
In comparison with the Carleman estimate in the continuous setting presented in \cite{ZXL:2025}, no additional truncation is applied to the weight function. In contrast, for the first Carleman estimate in Appendix~\ref{firstcarleman}, truncation is required to ensure the validity of the asymptotic properties of the weight functions established in \cite{AA:perez:2024}. 
\end{remark}
\subsection{Organization of the paper}
In Section \ref{sec:proofnullcontrollability}, we prove the $\phi$-null controllability for semi-discrete forward linear stochastic parabolic equations with source (Theorem \ref{lemma:theonullauxiliar}), by means of a minimization argument combined with the Carleman estimate. Section
\ref{sec:nullsemilinear} extends these results to the semilinear case via a fixed-point argument, completing the proof of Theorem \ref{theo:mainresult}. Section \ref{sec:comments} collects comments and concluding remarks, including a discussion of open questions. Finally, Appendix \ref{firstcarleman} establishes the new Carleman estimate for semi-discrete backward stochastic parabolic operators (Theorem \ref{theo:Carleman}); Appendix \ref{missingterms} provides the technical estimates for the cross-product terms; and Appendix \ref{sec:intermediateresult} contains the proof of the gradient localization.
\section{\texorpdfstring{$\phi$-}{}null Controllability for semi-discrete forward linear stochastic parabolic equations with source (proof of the Theorem~\ref{lemma:theonullauxiliar})}\label{sec:proofnullcontrollability}

\subsection{Minimization problem} Let $\mathcal{U}$ be an admissible control set given by 
\begin{equation*}
\begin{aligned}
    \mathcal{U}:=&\{ (u, U) \in L^2_{\mathbb{F}}(0,T; L_{h}^{2}(\mathcal{M}_0)) \times L^2_{\mathbb{F}}(0,T; L_{h}^{2}(\mathcal{M})) :\\
    &\E\int_0^T\int_{\mathcal{M}_0}s^{-3}\lambda^{-4}\xi^{-3}\rho^2|u|^2\,dt<\infty\,\quad\text{and}\quad\quad\E\int_{Q}s^{-2}\lambda^{-2}\xi^{-3}\rho^{2}|U|^2\,dt<\infty\}.
\end{aligned}
\end{equation*}
Then, we consider the following minimization problem
\begin{equation}\label{eq:ProblemMinimization}
    \inf_{(u,U)\in\mathcal{U}}J_{\epsilon}(u,U)\quad\quad\text{subject to the system \eqref{EQ:LFSPE}}
\end{equation}
where $J_\epsilon$ is defined as:
\begin{equation*}
\begin{aligned}
    J_{\epsilon}(u,U):=&\frac{1}{2}\E\int_{0}^T\int_{\mathcal{M}_0}s^{-3}\lambda^{-4}\xi^{-3}\rho^{2}|u|^2\,dt+\frac{1}{2}\E\int_{Q}s^{-2}\lambda^{-2}\xi^{-3}\rho^2|U|^2\,dt\\
    &+\frac{1}{2}\E\int_{Q}\rho^2|y|^2\,dt+\frac{1}{2\epsilon}\left.\int_{\mathcal{M}}|y|^2\right|_{t=T}.
\end{aligned}
\end{equation*}
We see that for $\epsilon>0$, the functional $J_\epsilon$ is continuous, strictly convex, and coercive over $\mathcal{U}$. Hence, the problem \eqref{eq:ProblemMinimization} admits a unique optimal control pairs $(u_\epsilon,U_{\epsilon})\in\mathcal{U}$, and the associated optimal solution for \eqref{EQ:LFSPE} is denoted by $y_\epsilon$. 

Our next goal is to determine an uniform bounds for the triple $(u_\epsilon,U_\epsilon,y_\epsilon)$. Using a duality argument, it follows from the Euler-Lagrange equation $J'_{\epsilon}(u_{\epsilon},U_{\epsilon})=0$ ($J'$ denotes the Fréchet derivative) that
the controls are given by
\begin{equation}\label{eq:formcontrols}  u_\epsilon=-s^3\lambda^4\xi^3r^2z_\epsilon\mathbbm{1}_{\mathcal{M}_0}\quad\text{and}\quad U_{\epsilon}=-s^2\lambda^2\xi^3r^2Z_\epsilon,
\end{equation}
where $(z_\epsilon,Z_\epsilon)$ satisfies the backward equation
\begin{equation}\label{eq:LBSPE}
\left\{\begin{array}{cc}
    \mathcal{P}^{\ast}z_{\epsilon}=\left(   \sum_{i=1}^{n}A_iD_i(a_{1i}z_{\epsilon})-a_2z_\epsilon-\rho^2y_\epsilon\right)\,dt+Z_{\epsilon}dB(t) &  \\
     z_\epsilon=0\quad\text{on}\,\,\partial Q,\quad \left.z_\epsilon\right|_{t=T}=\left.\frac{1}{\epsilon}y\right|_{t=T}\quad \text{in}\,\,\mathcal{M}, & 
\end{array}\right.
\end{equation}
and $y_\epsilon$ is the solution to system \eqref{EQ:LFSPE} associated with $(u_\epsilon,U_\epsilon)$.

Applying It\^{o}'s formula to the process $y_\epsilon z_\epsilon$, integrating over $Q$, then taking expectation and using that $y_{\epsilon}$ and $z_{\epsilon}$ satisfy \eqref{eq:LBSPE} and \eqref{EQ:LFSPE}, respectively; yield
\begin{align*}
    \E\int_{\mathcal{M}}&\left.y_{\epsilon}z_{\epsilon}\right|_{t=T}-\E\int_{\mathcal{M}}\left.y_\epsilon z_\epsilon\right|_{t=0}=\E\int_{Q}z_\epsilon dy_\epsilon+\E\int_{Q}y_\epsilon dz_\epsilon+\E\int_{Q}dz_\epsilon\,dy_\epsilon\\   
    =&\E\int_{Q}z_\epsilon\left(\sum_{i=1}^{n}D_i(\gamma_iD_iy_\epsilon)+a_{1i}A_iD_i(y_{\epsilon})+a_2y_\epsilon+v+\mathbbm{1}_{\mathcal{M}_0}u\right)\,dt\\
    &+\E\int_{Q}y_\epsilon\left(-\sum_{i=1}^{n}D_i(\gamma_iD_iz_\epsilon)+A_iD_i(a_{1i}z_\epsilon)-a_2z_\epsilon-\rho^2y_\epsilon\right)\,dt+\E\int_{Q}Z_{\epsilon}U_{\epsilon}\,dt.
    \end{align*}
Notice that using the discrete integration by parts \cite[Lemma 2.2]{LDOP-2021} and that $z_\epsilon=y_{\epsilon}=0$ on $\partial Q$, we obtain 
    \begin{align*}
\E\int_{\mathcal{M}}\left.y_{\epsilon}z_{\epsilon}\right|_{t=T}-\E\int_{\mathcal{M}}\left.y_\epsilon z_\epsilon\right|_{t=0}=&\E\int_{Q}z_\epsilon v\,dt+\E\int_{Q}\mathbbm{1}_{\mathcal{M}_0}u_{\epsilon}z_\epsilon\,dt-\E\int_{Q}\rho^2|y_\epsilon|^2\,dt\\
&+\E\int_{Q}Z_{\epsilon}U_{\epsilon}\,dt.
\end{align*}
Substituting the terminal value of $z_\epsilon$, and the characterization of the controls $(u_{\epsilon},U_{\epsilon})$ given by \eqref{eq:formcontrols} on the above equation, we rewrite it as
\begin{equation}\label{eq1:bound(u,U,z)}
    \begin{aligned}
    \frac{1}{\epsilon}\left.\E\int_{\mathcal{M}}|y_{\epsilon}|^2\right|_{t=T}+\E\int_{Q}\rho^2&|y_{\epsilon}|^2\,dt+\E\int_{0}^T\int_{\mathcal{M}_0}s^3\lambda^4\xi^3r^2|z_\epsilon|^2\,dt\\
    &+\E\int_{Q}s^2\lambda^2\xi^3r^2|Z_{\epsilon}|^2\,dt=\E\int_{Q}z_{\epsilon}v\,dt+\E\int_{\mathcal{M}}\left.y_{\epsilon}z_{\epsilon}\right|_{t=0}.
    \end{aligned}
\end{equation}
Now, applying Young's inequality on the right-hand side of \eqref{eq1:bound(u,U,z)} it follows that 
 \begin{equation}\label{ine:young:z:y(0)}
 \begin{aligned}
     \E\int_{Q}z_{\epsilon}&v\,dt+\E\left.\int_{\mathcal{M}}y_{\epsilon}z_{\epsilon}\right|_{t=0}\\
     &\leq \mu\left( \left.\E\int_{\mathcal{M}}\tau^{2}\lambda^{3}e^{2\lambda(6m+1)}e^{4\tau\varphi}|z_{\epsilon}|^{2}\right|_{t=0}+\E\int_{Q}s^3\lambda^4\xi^3r^2|z_\epsilon|^2\,dt\right)\\
     &+\frac{1}{4\mu}\left(\left.\E\int_{\mathcal{M}}\tau^{-2}\lambda^{-3}e^{-2\lambda(6m+1)}e^{-4\tau\varphi}|y_{\epsilon}|^{2}\right|_{t=0}+\E\int_{Q}s^{-3}\lambda^{-4}\xi^{-3}\rho^2|v|^2\,dt\right),
 \end{aligned}
 \end{equation}
 where the additional scaling terms are chosen according to the Carleman estimate \eqref{eq:InitialCarleman}. Thus, combining \eqref{ine:young:z:y(0)} with \eqref{eq1:bound(u,U,z)} we obtain
 \begin{equation}\label{eq2:bound(u,U,z)}
     \begin{aligned}
    \frac{1}{\epsilon}\E&\int_{\mathcal{M}}|y_{\epsilon}|^{2}|_{t=T}+\E\int_{Q}\rho^2|y_{\epsilon}|^{2}dt+\E\int_{0}^T\int_{\mathcal{M}_0}s^{3}\lambda^{4}\xi^{3}r^{2}|z_\epsilon|^{2}dt+\E\int_{Q}s^{2}\lambda^{2}\xi^{3}r^{2}|Z_{\epsilon}|^{2}dt\\
    &\leq  \mu\left( \left.\E\int_{\mathcal{M}}\tau^{2}\lambda^{3}e^{2\lambda(6m+1)}e^{4\tau\varphi}|z_{\epsilon}|^{2}\right|_{t=0}+\E\int_{Q}s^3\lambda^{4}\xi^{3}r^{2}|z_\epsilon|^{2}dt\right)\\
     &+\frac{1}{4\mu}\left(\left.\E\int_{\mathcal{M}}\tau^{-2}\lambda^{-3}e^{-2\lambda(6m+1)}e^{-4\tau\varphi}|y_{\epsilon}|^{2}\right|_{t=0}+\E\int_{Q}s^{-3}\lambda^{-4}\xi^{-3}\rho^2|v|^2\,dt\right).
    \end{aligned}
 \end{equation}
 On the other hand, thanks to Carleman estimate \eqref{eq:InitialCarleman} we can assert that exist $h_1>0$ sufficiently small, $\lambda_1,\tau_1>1$ such that the solution of the system \eqref{eq:LBSPE} verifies
\begin{equation}\label{eq:applyngCarlemaninitial}
     \begin{aligned}\E\left.\int_{\mathcal{M}}\tau^2\lambda^{3}e^{2\lambda(6m+1)}e^{4\tau \varphi}|z_{\epsilon}|^2\right|_{t=0}&+\E\int_{Q}s^{3}\lambda^4\xi^3r^2|z_\epsilon|^{2}dt\\
     &\leq  C\left(\E\int_{0}^T\int_{\mathcal{M}_0}s^3\lambda^4\xi^3r^2|z_\epsilon|^2\,dt+\E\int_{Q}r^2|\rho^2 y_{\epsilon}|^{2}dt\right.\\
     +\E&\left.\int_{Q}s^2\lambda^2\xi^2r^2|Z_\epsilon|^2\,dt+\frac{1}{(h\epsilon)^2}\E\left.\int_{\mathcal{M}}e^{2s(T)\varphi}|y_{\epsilon}|^2\right|_{t=T}\right)
     \end{aligned}
 \end{equation}
 for all $\tau>\tau_1$, $0<h<h_{1}$, $0<\delta<1/2$ and $s(t)h\leq \delta_0$.
 
 Taking $\mu=(2 C)^{-1}$ in \eqref{eq2:bound(u,U,z)}, and then using  \eqref{eq:applyngCarlemaninitial} we conclude that
 \begin{equation}\label{eq3:bound(u,U,z)}
     \begin{aligned}
    \frac{1}{\epsilon}\E&\left.\int_{\mathcal{M}}|y_{\epsilon}|^2\right|_{t=T}+\frac{1}{2}\E\int_{Q}\rho^2|y_{\epsilon}|^2\,dt+\frac{1}{2}\E\int_{0}^T\int_{\mathcal{M}_0}s^3\lambda^4\xi^3r^2|z_\epsilon|^2\,dt\\
    +\frac{1}{2}\E\int_{Q}&s^2\lambda^2\xi^3r^2|Z_{\epsilon}|^2\,dt\leq  \frac{1}{2(h\epsilon)^2}\E\left.\int_{\mathcal{M}}e^{2s(T)\varphi}|y_{\epsilon}|^2\right|_{t=T}\\
    &+\frac{ C}{2}\left(\left.\E\int_{\mathcal{M}}\tau^{-2}\lambda^{-3}e^{-2\lambda(6m+1)}e^{-4\tau\varphi}|y_{\epsilon}|^{2}\right|_{t=0}+\E\int_{Q}s^{-3}\lambda^{-4}\xi^{-3}\rho^2|v|^2\,dt\right).
    \end{aligned}
 \end{equation}
 
 \subsection{Weighted energy estimate for $D_{i}y_{\epsilon}$} 
 The task is now to find an appropriate uniform bound for $D_iy_{\epsilon}$, which will be achieved by performing a weight estimate for equation \eqref{EQ:LFSPE}.

 To begin, we apply It\^{o} formula to the process $s^{-2}\lambda^{-2}\xi^{-3}\rho^2|y_{\epsilon}|^2$. This yields
 \begin{equation}\label{eq1:boundD_iy}
 \begin{aligned}
   & \E\left.\int_{\mathcal{M}}(\delta T)^{-2m}\tau^{-2}\lambda^{-2}\xi^{-3}e^{2s(T) \varphi}|y_{\epsilon}|^2\right|_{t=T}-\E\left.\int_{\mathcal{M}}(2\tau)^{-2}\lambda^{-2}\xi^{-3}e^{4\tau \varphi}|y_{\epsilon}|^2\right|_{t=0}=\\
    &\E\int_{Q}\partial_t(s^{-2}\rho^2)\lambda^{-2}\xi^{-3}|y_{\epsilon}|^2\,dt+2\E\int_{Q}s^{-2}\lambda^{-2}\xi^{-3}\rho^{2}y_{\epsilon} dy_{\epsilon}+\E\int_{Q}s^{-2}\lambda^{-2}\xi^{-3}\rho^2|dy_{\epsilon}|^2.
\end{aligned}
\end{equation}
Recalling that $y_{\epsilon}$ satisfies \eqref{EQ:LFSPE}, we can rewrite the second term on the above equation as
\begin{align*}
    2\E\int_{Q}s^{-2}\lambda^{-2}\xi^{-3}\rho^2&y_{\epsilon}dy_{\epsilon}=2\sum_{i=1}^{n}\E\int_{Q}s^{-2}\lambda^{-2}\xi^{-3}\rho^2y_{\epsilon}[D_i(\gamma_iD_iy_\epsilon)+a_{1i}A_iD_i(y_{\epsilon})]\,dt\\
    &+2\E\int_{Q}a_{2}s^{-2}\lambda^{-2}\xi^{-3}\rho^2|y_{\epsilon}|^2\,dt+2\E\int_{Q}s^{-2}\lambda^{-2}\xi^{-3}\rho^2y_{\epsilon}v\,dt\\
    &+2\E\int_{0}^T\int_{\mathcal{M}_0}s^{-2}\lambda^{-2}\xi^{-3}\rho^2y_\epsilon u_{\epsilon}\,dt.
\end{align*}
Using integration by parts with respect to differential operator with the condition $y_\epsilon=0$ on $\partial Q$, and the discrete product rule,we can assert that
\begin{align*}
    2\E\int_{Q}s^{-2}\lambda^{-2}&\xi^{-3}\rho^2y_{\epsilon}[D_i(\gamma_iD_iy_\epsilon)+a_{1i}A_iD_i(y_{\epsilon})]dt\\
    =&-2\E\int_{Q_i^{*}}s^{-2}\lambda^{-2}\left[A_i(\xi^{-3}\rho^2)\gamma_i|D_iy_{\epsilon}|^2+\frac{1}{2}A_i(a_{1i}\xi^{-3}\rho^2)D_i(|y_\epsilon|^2)\right]dt\\
    &-2\E\int_{Q_i^{*}}s^{-2}\lambda^{-2}\left[\frac{1}{2}D_i(\xi^{-3}\rho^2)\gamma_iD_i(|y_{\epsilon}|^2)+D_i(a_{1i}\xi^{-3}\rho^2)|A_iy_\epsilon|^2\right]dt.
\end{align*}
Taking account that $r^2A_i(\xi^{-3}\rho^2)=\xi^{-3}+\mathcal{O}((sh)^{2})$, applying integration by parts with respect to differential operator on the above equation and combining these equalities, we can rewrite \eqref{eq1:boundD_iy} as

\begin{equation}\label{eq2:boundDiy}
    \begin{aligned}
        2&\sum_{i=1}^{n}\E\int_{Q_i^{*}}s^{-2}\lambda^{-2}(\xi^{-3}\rho^2+\mathcal{O}((sh)^{2})\gamma_i|D_iy_\epsilon|^2\,dt\\
        &=\sum_{i=1}^{n}\left(\E\int_{Q}s^{-2}\lambda^{-2}D_iA_i(a_{1i}\xi^{-3}\rho^2)|y_\epsilon|^2\,dt+\E\int_{Q}s^{-2}\lambda^{-2}D_i(D_i(\xi^{-3}\rho^2)\gamma_i)|y_{\epsilon}|^2\;dt\right.\\
    &\left.-\E\int_{Q_i^{*}}D_i(a_{1i}\xi^{-3}\rho^2)|A_iy_\epsilon|^2\,dt\right)+2\E\int_{Q}a_{2}s^{-2}\lambda^{-2}\xi^{-3}\rho^2|y_{\epsilon}|^2\,dt\\
    &+2\E\int_{Q}s^{-2}\lambda^{-2}\xi^{-3}\rho^2y_{\epsilon}v\,dt+2\E\int_{0}^T\int_{\mathcal{M}_0}s^{-2}\lambda^{-2}\xi^{-3}\rho^2y_\epsilon u_{\epsilon}\,dt\\
    &-\E\left.\int_{\mathcal{M}}(\delta T)^{-2m}\tau^{-2}\lambda^{-2}\xi^{-3}e^{2s(T) \varphi}|y_{\epsilon}|^2\right|_{t=T}+\E\left.\int_{\mathcal{M}}(2\tau)^{-2}\lambda^{-2}\xi^{-3}e^{4\tau \varphi}|y_{\epsilon}|^2\right|_{t=0}\\
    &+\E\int_{Q}\partial_t(s^{-2}\rho^2)\lambda^{-2}\xi^{-3}|y_{\epsilon}|^2\,dt+\E\int_{Q}s^{-2}\lambda^{-2}\xi^{-3}\rho^2|U_\epsilon|^2\,dt.
    \end{aligned}
\end{equation}
where used that $y_\epsilon=0$ on $\partial Q$ again. We now proceed to find a upper bound for the first three terms on the right-hand side of the previous equation. Using that $|r^2D_i(a_{1i}\xi^{-3}\rho^2)|\leq  Cs\lambda\xi^{-2}$, $|r^2D_iA_i(a_{1i}\xi^{-3}\rho^2)|\leq  Cs\lambda\xi^{-2}$ and $|r^2D_i(\gamma_iD_i(\xi^{-3}\rho^2))|\leq  Cs^2\lambda^2\xi^{-1}$, we have
\begin{align*}
    \E\int_{Q}s^{-2}\lambda^{-2}D_iA_i(a_{1i}\xi^{-3}\rho^2)|y_\epsilon|^2\,dt+&\E\int_{Q}s^{-2}\lambda^{-2}D_i(D_i(\xi^{-3}\rho^2)\gamma_i)|y_{\epsilon}|^2\;dt\\
    -\E\int_{Q_i^{*}}D_i(a_{1i}\xi^{-3}\rho^2)|A_iy_\epsilon|^2\,dt\leq& C\left(\E\int_{Q}(s^{-1}\lambda^{-1}\xi^{-2}\rho^2+\xi^{-1}\rho^2)|y_{\epsilon}|^2\,dt\right.\\
    &\left.+\E\int_{Q_i^{\ast}}s^{-1}\lambda^{-1}\xi^{-2}\rho^2|A_iy_\epsilon|^2\right).
\end{align*}
Moreover, using the fact $|A_iy_{\epsilon}|^2\leq A_i|y_{\epsilon}|^2$, integration by parts with respect to average operator and the condition that $y_{\epsilon}=0$ on $\partial Q$, we get
\begin{align*}
    \E\int_{Q}s^{-2}\lambda^{-2}D_iA_i(a_{1i}\xi^{-3}\rho^2)|y_\epsilon|^2&\,dt+\E\int_{Q}s^{-2}\lambda^{-2}D_i(D_i(\xi^{-3}\rho^2)\gamma_i)|y_{\epsilon}|^2\;dt\\
    -\E\int_{Q_i^{*}}D_i(a_{1i}\xi^{-3}\rho^2)|A_iy_\epsilon|^2\,dt\leq C&\left(\E\int_{Q}(s^{-1}\lambda^{-1}\xi^{-2}+\xi^{-1}+\mathcal{O}((sh)^2))\rho^2|y_{\epsilon}|^2\,dt\right),
\end{align*}
where used that $r^2A_i(s^{-1}\lambda^{-1}\xi^{-2}\rho^2)=s^{-1}\lambda^{-1}\xi^{-2}+\mathcal{O}((sh)^2)$. Substituting the above inequality into \eqref{eq2:boundDiy} and applying the Young inequality on fourth and fifth on the right side-hand of \eqref{eq2:boundDiy}, we have

\begin{equation}\label{eq3:boundDiy}
    \begin{aligned}
        &\sum_{i=1}^{n}\int_{Q_i^{*}}s^{-2}\lambda^{-2}(\xi^{-3}+\mathcal{O}((sh)^{2})\gamma_i\rho^2|D_iy_\epsilon|^2\,dt\\
    &\leq C\left(\E\int_{Q}(s^{-1}\lambda^{-1}\xi^{-2}+\xi^{-1}+\mathcal{O}((sh)^2)+s^{-2}\lambda^{-2}\xi^{-3})\rho^2|y_{\epsilon}|^2\,dt\right)\\
    &+2\E\int_{Q}\xi^{-3}\rho^2|y_\epsilon|^2\;dt+\E\int_{Q}s^{-4}\lambda^{-4}\xi^{-6}\rho^2|v|^2\,dt\\
    &+\E\int_{0}^T\int_{\mathcal{M}_0}s^{-4}\lambda^{-4}\xi^{-6}\rho^2| u_{\epsilon}|^2\,dt+\E\left.\int_{\mathcal{M}}(2\tau)^{-2}\lambda^{-2}\xi^{-3}e^{4\tau \varphi}|y_{\epsilon}|^2\right|_{t=0}\\
    &+\E\int_{Q}\partial_t(s^{-2}\rho^2)\lambda^{-2}\xi^{-3}|y_{\epsilon}|^2\,dt+\E\int_{Q}s^{-2}\lambda^{-2}\xi^{-3}\rho^2|U_\epsilon|^2\,dt.
    \end{aligned}
\end{equation}
Now, we focus on the penultimate term on the right-hand side of the above inequality. Note that $1\leq \theta\leq 2$ on $[0,T/4]$, $\varphi$ is a negative function over $Q$, and $\theta_{t}=-4T^{-1}\theta ^{-1}s\lambda^2\sigma(1-4tT^{-1})^{\sigma-1}e^{\lambda(6m-4)}\leq 0$, we have
\begin{align*}
\E\int_{0}^{T/4}\int_{\mathcal{M}}&\partial_t(s^{-2}\rho^2)\lambda^{-2}\xi^{-3}|y_{\epsilon}|^2\,dt\\
&=\E\int_{0}^{T/4}\int_{\mathcal{M}}\frac{\theta_t}{\theta}\left(-2s^{-2}\rho^2-2s^{-2}\lambda\rho^2\varphi\right)\lambda^{-2}\xi^{-3}|y_{\epsilon}|^2\,dt\\
    \leq& C\E\int_0^{T/4}\int_{\mathcal{M}} s^{-1}\xi^{-3}\rho^2|y_{\epsilon}|^2\,dt\leq  C\E\int_{Q}\rho^2|y_{\epsilon}|^2\,dt.
\end{align*}
On the other hand, using the fact of $|\theta_t|\leq C|\theta|^2$ for all $t\in [T/4,T]$, one can obtain that
\begin{align*}
    \left|\E\int_{T/4}^{T}\int_{\mathcal{M}}\partial_t(s^{-2}\rho^2)\lambda^{-2}\xi^{-3}|y_{\epsilon}|^2\,dt\right|\\
    \leq \E\int_{T/4}^{T}\int_{\mathcal{M}}\frac{2|\theta_t|}{\theta}&\left|s^{-2}\rho^2-2s^{-2}\lambda\rho^2\varphi\right|\lambda^{-2}\xi^{-3}|y_{\epsilon}|^2\,dt\\
    \leq C\E\int_{T/4}^{T}&\int_{\mathcal{M}} s^{-2}\lambda^{-1}\xi^{-3}\rho^2|y_{\epsilon}|^2\,dt\leq  C\E\int_{Q}\rho^2|y_{\epsilon}|^2\,dt.
\end{align*}
Therefore, 
\begin{equation}\label{eq:estimatiationpt2}
\int_{Q}\partial_t(s^{-2}\rho^2)\lambda^{-2}\xi^{-3}|y_{\epsilon}|^2\,dt\leq  C\int_{Q}\rho^2|y_{\epsilon}|^2\,dt.
\end{equation}
Combining \eqref{eq:estimatiationpt2} with \eqref{eq3:boundDiy}, yields
\begin{equation}\label{eq4:boundDiy}
    \begin{aligned}
        \sum_{i=1}^{n}&\E\int_{Q_i^{*}}s^{-2}\lambda^{-2}(\xi^{-3}+\mathcal{O}((sh)^{2})\gamma_i\rho^2|D_iy_\epsilon|^2\,dt\\
    &\leq C\E\int_{Q}(s^{-1}\lambda^{-1}\xi^{-2}+\xi^{-1}+\mathcal{O}((sh)^2)+s^{-2}\lambda^{-2}\xi^{-3}+1)\rho^2|y_{\epsilon}|^2\,dt\\
    &+\E\int_{Q}s^{-4}\lambda^{-4}\xi^{-6}\rho^2|v|^2\,dt+\E\int_{0}^T\int_{\mathcal{M}_0}s^{-4}\lambda^{-4}\xi^{-6}\rho^2| u_{\epsilon}|^2\,dt\\
    &+\E\left.\int_{\mathcal{M}}(2\tau)^{-2}\lambda^{-2}\xi^{-3}e^{4\tau \varphi}|y_{\epsilon}|^2\right|_{t=0}+\E\int_{Q}s^{-2}\lambda^{-2}\xi^{-3}\rho^2|U_\epsilon|^2\,dt.
    \end{aligned}
\end{equation}
From the fact $\mathcal{O}((sh))^2\leq \varepsilon_0$, $\xi,\tau,\lambda>1$, it follows that
\begin{equation}
     \begin{aligned}
        \sum_{i=1}^{n}&\E\int_{Q_i^{*}}s^{-2}\lambda^{-2}\xi^{-3}\rho^2|D_iy_\epsilon|^2\,dt\\
    &\leq C\left(\E\int_{Q}\rho^2|y_{\epsilon}|^2\,dt+\E\int_{Q}s^{-4}\lambda^{-4}\xi^{-3}\rho^2|v|^{2}dt+\E\int_{0}^T\int_{\mathcal{M}_0}s^{-4}\lambda^{-4}\xi^{-3}\rho^2| u_{\epsilon}|^{2}dt\right.\\
    &\left.+\E\left.\int_{\mathcal{M}}\tau^{-2}\lambda^{-2}\xi^{-3}e^{4\tau \varphi}|y_{\epsilon}|^2\right|_{t=0}+\E\int_{Q}s^{-2}\lambda^{-2}\xi^{-3}\rho^2|U_\epsilon|^2\,dt\right).
    \end{aligned}
\end{equation}
The definition of controls \eqref{eq:formcontrols}, implies that
\begin{equation}\label{eq4:bound_Diy}
     \begin{aligned}
        \sum_{i=1}^{n}&\E\int_{Q_i^{*}}s^{-2}\lambda^{-2}\xi^{-3}\rho^2|D_iy_\epsilon|^2\,dt\\
    &\leq C\left(\E\int_{Q}\rho^2|y_{\epsilon}|^2\,dt+\E\int_{0}^T\int_{\mathcal{M}_0}s^{2}\lambda^{4}\xi^{3}r^2| z_\epsilon|^2\,dt+\E\int_{Q}s^{2}\lambda^{2}\xi^{3}r^2|Z_\epsilon|^2\,dt\right.\\
    &\left.+\E\left.\int_{\mathcal{M}}\tau^{-2}\lambda^{-2}\xi^{-3}e^{4\tau \varphi}|y_{\epsilon}|^2\right|_{t=0}+\E\int_{Q}s^{-4}\lambda^{-4}\xi^{-3}\rho^2|v|^2\,dt\right).
    \end{aligned}
\end{equation}
Therefore, 
\begin{equation}\label{eq5:boundDiy}
     \begin{aligned}
        \sum_{i=1}^{n}\E\int_{Q_i^{*}}s^{-2}\lambda^{-2}&\xi^{-3}\rho^2|D_iy_\epsilon|^2\,dt\leq C\left(\frac{1}{(h\epsilon)^2}\E\left.\int_{\mathcal{M}}e^{2s(T)\varphi}|y_\epsilon|^2\right|_{t=T}\right.\\
        &\left.+\E\left.\int_{\mathcal{M}}\tau^{-2}\lambda^{-2}\xi^{-3}e^{4\tau \varphi}|y_{\epsilon}|^2\right|_{t=0}+\E\int_{Q}s^{-4}\lambda^{-4}\xi^{-3}\rho^2|v|^2\,dt\right),
    \end{aligned}
\end{equation}
the last inequality being a consequence of the comparison of \eqref{eq4:bound_Diy} with \eqref{eq3:bound(u,U,z)}. 
\subsection{Combining the previous estimates}
Combining the estimates \eqref{eq3:bound(u,U,z)} with \eqref{eq5:boundDiy}, we get
\begin{equation}\label{eq6:bound(u,U,z)}
\begin{split}
&\E\int_{\mathcal{M}}\left(\frac{1}{\epsilon}-\frac{ C}{2(h\epsilon)^{2}}e^{2s(T)\varphi}\right)|y_{\epsilon}(T)|^2+\sum_{i=1}^{n}\E\int_{Q_i^{\ast}}s^{-2}\lambda^{-2}\xi^{-3}\rho^2|D_iy_{\epsilon}|^2\,dt\\
&+\E\int_{Q}\rho^2|y_{\epsilon}|^2\,dt+\E\int_{0}^T\int_{\mathcal{O}_0\cap\mathcal{M}}s^{-3}\lambda^{-4}\xi^{-3}\rho^{2}|u_{\epsilon}|^2\,dt+\E\int_{Q}s^{-2}\lambda^{-2}\xi^{-3}\rho^2|U_{\epsilon}^2|\,dt\\
\leq& C\left(\E\int_{Q}s^3\lambda^4\xi^3r^2|z|^2\,dt+\E\int_{Q}s^2\lambda^2\xi^3r^2|Z|^2\,dt\right.\\
&\left.+\E\int_{\mathcal{M}}\tau^{-2}\lambda^{-3}e^{-2\lambda(6m+1)}e^{-4\tau\varphi}|y_{\epsilon}(0)|^2\right)        
    \end{split}
\end{equation}
Now, we are going to focus on the first term of the previous equation: Notice that $\varphi\leq -(\lambda -1)e^{6\lambda (m+1)}$ and taking $\epsilon:=\mathcal{E}_{\lambda,h}= Ch^{-2}e^{-2s(T)(\lambda-1)e^{6\lambda(m+1)}}$, we have the following. 
\begin{equation*}
    \begin{split}
        \E\int_{\mathcal{M}}\bigg(\frac{1}{\epsilon}-\frac{ C}{2(h\epsilon)^2}e^{2s(T)\varphi}\bigg)|y_{\epsilon}(T)|^2\geq \frac{1}{2\mathcal{E}_{\lambda,h}}\E\int_{\mathcal{M}}|y_{\epsilon}(T)|^2.
    \end{split}
\end{equation*}
Therefore, combining the above inequality and the inequality \eqref{eq6:bound(u,U,z)}, we can conclude the desired result.
\section{Controllability of semidiscrete forward semilinear SPDE (Theorem \ref{theo:mainresult})}\label{sec:nullsemilinear}
Consider the following subspace 
\begin{align*}
    \mathcal{D}_{\tau,\lambda}:=\left\{v\in L^2_{\mathbb{F}}(0,T;L_{h}^2(\mathcal{M}))\mid\|v\|^{2}_{\mathcal{D}_{\tau,\lambda}}:= \E\int_{Q}s^{-3}\lambda^{-4}\xi^{-3}\rho^2|v|^2<\infty\right\},
\end{align*}
which is Banach with the canonical norm. Given $v\in L^2_{\mathds{F}}(0,T;L_{h}^2(\mathcal{M})$, we consider the following controlled system
\begin{equation}\label{eq:forwardfixed}
    \left\{\begin{array}{lr}
        \mathcal{P}y=(v+\mathbbm{1}_{\mathcal{O}_0}u)\,dt+U\,dB(t)\text{ in }Q&\\
        y=0\text{ on }\partial Q, \quad\quad y(0)=y_0 \text{ in }\mathcal{M}.
    \end{array}\right.
\end{equation}
Notice that this system is a particular  case of \eqref{EQ:LFSPE}. Then, there exist control functions $(u,U)\in L^2_{\mathds{F}}(0,T;L_{h}^2(\mathcal{M}_0))\times L^2_{\mathds{F}}(0,T;L_{h}^2(\mathcal{M}))$ and the corresponding solution $y$ to \eqref{eq:forwardfixed} satisfying \eqref{eq:boundyz(T)}-\eqref{eq:estimateCarlmenz} follow from Theorem \ref{lemma:theonullauxiliar}.

For each $v\in\mathcal{D}_{\tau,\lambda}$, let $(y_{v,}u_{v},U_{v})$ be the optimal triple for the minimization problem \eqref{eq:ProblemMinimization} applied to system \eqref{eq:forwardfixed} with source $v$, as given by Theorem \ref{lemma:theonullauxiliar}. Define $\mathcal{G}:\mathcal{D}_{\tau,\lambda}\rightarrow L^{2}_{\mathcal{F}}(0,T;L^{2}_{h}(\mathcal{M}))$ by  $\mathcal{G}(v):=F_{1}(\omega,t,x,y_{v},\nabla_{h}y_{v})$.\\
Let us first prove that if $v\in\mathcal{D}_{\tau,\lambda}$ then $\mathcal{D}_{\tau,\lambda}(v)\in\mathcal{D}_{\tau,\lambda}$. Indeed, by assumption (A2), the Lipschitz condition gives
\begin{equation}\label{ine:G:estimate}
\begin{aligned}
    \|\mathcal{G}(v)\|_{\mathcal{D}_{\tau,\lambda}}^2=& \E\int_{Q}s^{-3}\lambda^{-4}\xi^{-3}\rho^2|F_{1}(w,t,x,y_{v},\nabla_{h}y_{v})|^2\,dt\\
    \leq&2L^{2}_1\E\int_{Q}s^{-3}\lambda^{-4}\xi^{-3}\rho^2\left(|y_{v}|^2+\sum_{i=1}^{n}|A_iD_iy_{v}|^2\right)\,dt
\end{aligned}
\end{equation}
Let us focus on the integral with term $|y_{v}|^{2}$. Since $s^{-2}\lambda^{-3}\xi^{-3}\leq 1$ and $s^{-1}\lambda^{-2}\leq \tau^{-1}\lambda^{-2}$ for $s,\xi\geq 1$ we have
\begin{equation}\label{ine:estimate:yv}
	\begin{aligned}
&	\E\int_{Q}s^{-3}\lambda^{-4}\xi^{-3}\rho^{2}|y_{v}|^{2}dt\leq \E\int_{Q}\rho^{2}|y_{v}|^{2}dt\\
	&\leq C\tau^{-1}\lambda^{-1}\left( \E\int_{Q}s^{-3}\lambda^{-4}\xi^{-3}\rho^{2}|v|^{2}dt+\int_{\mathcal{M}}\tau^{-1}\lambda^{-2}e^{-2\lambda(6m+1)}e^{-4\tau\varphi}|y_{0}|^{2}dt\right).
	\end{aligned}
	\end{equation}
	where the last line follows from \eqref{eq:estimateCarlmenz}.\\
In turn, applying $|A_iD_iy_{v}|^2\leq A_i|D_iy_{v}|^2$ and an integration by parts with respect to the average operator we have
\begin{equation}
\begin{aligned}
&\E\int_{Q} s^{-3}\lambda^{-4}\xi^{-3}\rho^{2}|A_{i}D_{i}y_{v}|^{2}dt\leq\E\int_{Q} s^{-3}\lambda^{-4}\xi^{-3}\rho^{2}A_{i}(|D_{i}y_{v}|^{2})dt\\
&\leq\E\int_{Q_i^{\ast}}s^{-3}\lambda^{-4}A_i(\xi^{-3}\rho^2)|D_iy_{v}|^2\,dt-\frac{h}{4}\int_{\partial_i Q}s^{-3}\lambda^{-4}\xi^{-3}\rho^2t_{r}^{i}(|D_iy_{v}|^2)\,dt\\
&\leq\E\int_{Q_i^{\ast}}s^{-3}\lambda^{-4}A_i(\xi^{-3}\rho^2)|D_iy_{v}|^2\,dt,
\end{aligned}
\end{equation}
where in the last line we have dropped the boundary integral since is negative.
Now, from the asymptotic expansion $r^2A_{i}(s^{-3}\lambda^{-4}\xi^{-3}\rho^2)=s^{-3}\lambda^{-4}\xi^{-3}+\mathcal{O}_{\lambda}((sh)^2)$, see \cite{AA:perez:2024},  it follows that $A_{i}(s^{-3}\lambda^{-4}\xi^{-3}\rho^2)=s^{-3}\lambda^{-4}\xi^{-3}\rho^{2}(1+\mathcal{O}_{\lambda}((sh)^{2}))\leq Cs^{-3}\lambda^{-4}\xi^{-3}\rho^{2}$ provided $sh<1$. Thus, using \eqref{eq:estimateCarlmenz}, we get
\begin{equation}\label{ine:estimate:Dv}
	\begin{aligned}
	&\E\int_{Q} s^{-3}\lambda^{-4}\xi^{-3}\rho^{2}|A_{i}D_{i}y_{v}|^{2}dt\leq \tau^{-1}\lambda^{-2}\E\int_{Q} s^{-2}\lambda^{-2}\xi^{-3}\rho^{2}|D_{i}y_{v}|^{2}dt\\
	&\leq C\tau^{-1}\lambda^{-2}\left(\E\int_{Q}s^{-2}\lambda^{-3}\rho^{2}|v|^{2}+\E\int_{\mathcal{M}}\tau^{-1}\lambda^{-1}e^{-2\lambda(6m+1)}e^{-4\tau\varphi}|y_{0}|\right)
	\end{aligned}
	\end{equation}
	Combing \eqref{ine:estimate:yv} and \eqref{ine:estimate:Dv} in the right-hand side of \eqref{ine:G:estimate} we obtain
	\begin{equation}\label{ine:G:mapping}
		\begin{aligned}
			    \|\mathcal{G}(v)\|_{\mathcal{D}_{\tau,\lambda}}^2&\leq 2L^{2}_{1}C\tau^{-1}\lambda^{-2}\left( \|v\|^{2}_{\mathcal{D}_{\tau,\lambda}}+\E\int_{\mathcal{M}}\tau^{-1}\lambda^{-2}e^{-2\lambda(6m+1)}e^{-4\tau\varphi}|y_{0}|^{2}\right)\\
			    &<\infty,
			\end{aligned}
		\end{equation}
which proves that $\mathcal{G}(v)\in \mathcal{D}_{\tau,\lambda}$.\\
Our next task is to prove that map $\mathcal{G}$ is a contration. Let $v_{1},v_{2}\in\mathcal{D}_{\tau,\lambda}$ with solutions $y_{1},y_{2}$ and controls $(u_{1},U_{1})$, $(u_{2},U_{2})$. Set $\overline{y}:=y_{1}-y_{2}$, $\overline{v}:=v_{1}-v_{2}$, $\overline{u}:=u_{1}-y_{2}$, and $\overline{U}:=U_{2}-U_{1}$. By linearity $\overline{y}$ verifyies
\begin{equation}
	\mathcal{P}\overline{y}=(\overline{v}+\mathbbm{1}_{\mathcal{M}_{0}})dt+\overline{U}dB(t),\quad \overline{y}=0 \text{ on }\partial Q, \quad \overline{y}(0)=0.
	\end{equation} 
Thanks to assumption (A2) it follows that
\begin{align*}
    \|\mathcal{G}(v_{1})-\mathcal{G}(v_{2})\|^{2}_{\mathcal{D}_{\tau,\lambda}}=&\E\int_{Q}s^{-3}\lambda^{-4}\xi^{-3}\rho^2|F_{1}(\cdot,y_{1},\nabla_{h}y_{1})-F_{1}(\cdot,y_{2},\nabla_{h}y_{2})|^2\,dt\\
    \leq&2L_{1}^{2}\E\int_{Q}s^{-3}\lambda^{-4}\xi^{-3}\rho^2\left(|\overline{y}|^2+\sum_{i=1}^{n}|A_iD_i\overline{y}|^2\right)\,dt
\end{align*}
Then, using the same argument employed to show \eqref{ine:G:mapping} since the right-hand side has the same structure, we obtain 
\begin{align*}
   \|\mathcal{G}v_{1}-\mathcal{G}v_{2}\|^{2}_{\mathcal{D}_{\tau,\lambda}}&\leq 2L_{1}^{2}C\tau^{-1}\lambda^{-2}\|v_{1}-v_2\|_{\mathcal{D}_{\tau,\lambda}}^{2}. 
\end{align*}
Thus, by choosing $\tau>1$ sufficiently large  such that $\tau^{1/2}\lambda>\sqrt{2C}L_{1}$ we deduce that $\mathcal{G}$ is a contraction mapping.\\
Hence, by the Banach contraction mapping Theorem there exists a unique $\hat{v}\in\mathcal{D}_{\tau,\lambda}$ such that $\mathcal{G}(\hat{v})=\hat{v}$. This fixed point satisfies $\hat{v}=F_{1}(\cdot,\hat{y},\nabla_{h}\hat{y})$, where $\hat{y}$ solves \eqref{eq:forwardfixed}. Therefore $\hat{y}$ solves the reduced semilinear system \eqref{system:nonlinear}. Moreover, using \eqref{ine:G:mapping} with $v=\hat{v}$ 
\begin{equation}
	\| \hat{v}\|^{2}_{\mathcal{D}_{\tau,\lambda}}\leq 2CL_{1}^{2}\tau^{-1}\lambda^{-2}\left( \|\hat{v}\|^{2}_{\mathcal{D}_{\tau,\lambda}}+\E\int_{\mathcal{M}}\tau^{-1}\lambda^{-2}e^{-2\lambda(6m+1)}e^{-4\tau\varphi}|y_{0}|^{2} \right).
	\end{equation}
Notice that $2CL_{1}^{2}\tau^{-1}\lambda^{2}<1$ due to the contraction condition. Therefore
\begin{equation}\label{ine:bounded:control}
	\| \hat{v}\|^{2}_{\mathcal{D}_{\tau,\lambda}}\leq C\E\int_{\mathcal{M}}e^{-4\tau\varphi}|y_{0}|^{2}dt.
	\end{equation}
Moreover, \eqref{eq:boundyz(T)} from Theorem \ref{lemma:theonullauxiliar} also gives  
\begin{equation}\label{eq:inequality y(t)}
    \begin{split}
        \E \int_{\mathcal{M}}|\hat{y}(T)|^2\leq& \mathcal{E}_{\lambda,h}\bigg(\|v\|^{2}_{\mathcal{D}_{\tau,\lambda}}+\E\int_{\mathcal{M}}\tau^{-1}\lambda^{-2}e^{-2\lambda(6m+1)}e^{-4\tau\varphi}|y_0|^2\bigg)\\
        \leq &\mathcal{E}_{\lambda,h}\E\int_{\mathcal{M}}e^{-4\tau\varphi}|y_{0}|^{2}dt.
    \end{split}
\end{equation}
where in the last line we have used \eqref{ine:bounded:control} and part of the weighted factor is less than one. Then, from the definition of $\varphi$ we have $e^{-4\tau\varphi}\leq e^{C\tau}$ with $C=C(\lambda)$. Moreover, recalling that $\mathcal{E}_{\lambda,h}:=C h^{-2}e^{-2s(T)(\lambda-1)e^{6\lambda(m+1)}}$ with $s(T)=\tau(\delta T)^{-m}$. It follows that 
\begin{equation*}
\mathcal{E}_{\lambda,h}e^{C\lambda}=Ch^{-2}\exp\left( 2\tau[C/2-(\delta T)^{-m}(\lambda-1)e^{6\lambda(m+1)}]\right)
\end{equation*}
Now, we choose $\delta_{0}>0$ small enough such that for $0<\delta<\delta_{0}$ we have $C/2-(\delta T)^{-m}(\lambda-1)e^{6\lambda(m+1)}\leq -C'(\delta T)^{-m}$. Thus, for $\delta <\delta_{0}$ 
\begin{equation}
\mathcal{E}_{\lambda,h}e^{C\lambda}\leq Ch^{-2}e^{-C'\tau(\delta T)^{m}}
\end{equation}
Our last task is to connect $h$ with $\delta$. Notice that setting $h_{1}:= \epsilon_{0}(\delta_{0}T)^{m}/\tau$ holds $\tau(\delta_{0}T)^{-m}h_{1}=\epsilon_{0}$. Then, for $h\leq \min\{h_{0},h_{1}\}$ we set $\delta:=\left( \frac{h}{h_{1}}\right)^{1/m}\delta_{0}$ which verifies $\delta\leq \delta_{0}$, $\tau(\delta T)^{-m}h=\epsilon_{0}$ and $(\delta T)^{-m}=\frac{\epsilon_{0}}{\tau h}$. These conditions, applied in \eqref{eq:inequality y(t)}, allows us to obtain 
\begin{equation*}
     \E \int_{\mathcal{M}}|\hat{y}(T)|^2
     \leq \frac{ C}{h^2}e^{-C\epsilon_{0}/h]}\E\int_{\mathcal{M}}|y_0|^2,
\end{equation*}
Since $\lim_{h\rightarrow 0^{+}}h^{-2}e^{-1/h}=0$, the polynomial factor $h^{-2}$ can be absorbed. Therefore
\begin{equation*}
    \mathbb{E} \int_{\mathcal{M}} |\hat{y}(T)|^2 
    \leq  C e^{-C/h} 
    \mathbb{E} \int_{\mathcal{M}} |y_0|^2,
\end{equation*}
which completes the proof of Theorem \ref{theo:mainresult}.
\section{Comments and concluding remarks}\label{sec:comments}
In this work we have established that for any function $\phi:(0,\infty)\rightarrow(0,\infty)$ satisfying $\lim_{h\rightarrow 0}\phi(h)=0$ and $\liminf_{h\rightarrow 0}\phi(h)/e^{-\kappa h^{-1}}>0$, there exist uniformly control $(u,U)$ such that the solution \eqref{system:nonlinear} satisfies
\begin{equation*}
    \E\int_{\mathcal{M}}|y(T)|^{2}\leq C\phi(h)\E \int_{\mathcal{M}}|y_{0}|^{2},
\end{equation*}
where the constants $C,\kappa>0$ are independent of $h$ and $y_{0}$. The strategy used in this work could be applied to study similar system or similar related controllability results. Let us describe two possible future direction. \\
A natural next step is the fully discrete case, where time is also discretized, for instance an implicit Euler scheme with step 
\begin{equation*}
\frac{y^{k+1} - y^k}{\Delta t}
-
\sum_{i=1}^n D_i\bigl(\gamma_i D_i y^{k+1}\bigr)
=
F_1(y^k,\nabla_h y^k) + \mathbbm{1}_{\mathcal{M}_0} u^k + \frac{\Delta B_k}{\Delta t} U^k,
\end{equation*}
with $\Delta B_k = B(t_{k+1}) - B(t_k) \sim \mathcal{N}(0,\Delta t)$.

In the deterministic setting, \cite{BHLR:2011} established $\phi$-null controllability assuming a partial Lebeau-Robbiano inequality, by proving fully discrete Carleman estimate \cite{GC-HS-2021} and \cite{LMPZ-2023} studied the one-dimensional case with Dirichlet and dynamic boundary condition, respectively. In arbitrary dimension, $\phi$-null control result for fully discrete parabolic operators is obtained in \cite{AA:perez:2024}. However, the fully discrete stochastic case remains open. A possible strategy could be to mimic the penalized variational approach applied in this work. To this end, it a first task should be to prove a fully discrete Carleman estimate for the corresponding fully discrete backward stochastic system. As is reported in the deterministic case \cite{GC-HS-2021, LMPZ-2023} and \cite{AA:perez:2024} a CFL-type condition is expected. \\
Another future direction could be to consider system \eqref{eq:systemnonlinear}, but with initial data
\begin{equation*}
y_0 + \sigma \hat{y}_0    
\end{equation*}
where $\hat{y}_0$ is unknown, and to study the existence of the existence of a control pair $(u,U)$ such that the observation functional
\begin{equation*}
\Phi_\sigma := \mathbb{E} \int_0^T \int_O |y_\sigma|^2 \, dt,    
\end{equation*}
where $O \subset \mathcal{M}$ is an observation region, verifies
\begin{equation*}
\left.\frac{\partial \Phi_\sigma}{\partial \sigma}\right|_{\sigma=0} = 0    
\end{equation*}
for all $\hat{y}_0$, this is known as insensitizing controllability. As shown in the deterministic continuous setting \cite{DT:2000}, this reduces to null controllability of the backward component in a forward--backward cascade, under a geometric condition of the control and observation regions. Since in \cite{LPP:2025} is given an example that the null-controllability is false in general is not expected to obtain this type of properties. However, it could be possible to study a relaxed insesitizing controllability as is proved in the semi-discrete deterministic case \cite{BHSDT:2019}. The known strategy uses Carleman estimates for the Forward and Backward system, and since in \cite{LPP:2026} are obtained Carleman estimates for Forward semi-discrete stochastic parabolic operator, it could be possible to extend into the semi-discrete stochastic framework insensitizing controllability results at least for the linear case.

\section{Acknowledgment}
 R. Lecaros was partially supported by FONDECYT (Chile) Grant 1260574. A. A. P\'erez acknowledges the support of Vicerrector\'ia de Investigaci\'on y postgrado, Universidad del B\'io-B\'io, project IN2450902 and FONDECYT Grant 11250805. M. F. Prado gratefully acknowledges the support from the Institutional Scholarship Fund of the University of Valpara\'iso (FIB-UV), proyecto interno PI LIR 25 14 and Programa de Iniciación a la Investigación Científica N° 049/2025.
\appendix
\section{Proof of Theorem \ref{theo:Carleman}}\label{firstcarleman}

We note that $\theta(t)\in [1,2]$ for $t\in[0,T/2]$. The parameter $\delta$ is chosen so that $0<\delta<\tfrac{1}{2}$ in order to avoid singularity at time $T$, and in this case $\theta(t)\in [1,(\delta T)^{-m}]$ for $t\in [T/2,T]$. Moreover, for the first time, derivative holds $|\theta_t(t)|\leq 4T^{-1}\sigma$ for $t\in [0,T/2)$, and $|\theta_t(t)|\leq C|\theta(t)|^2$ for $t\in [T/2,T]$. Finally, for the second derivative we have $|\theta_{tt}(t)|\leq (4T^{-1})^2\sigma(\sigma-1)$ for $t\in [0,T/2)$ and $|\theta_{tt}(t)|\leq C(\theta(t))^3$ for $t\in [T/2,T]$.
\begin{proof}  For the sake of presentation, we split the proof into three steps: First, we write the conjugate operator into two parts, and an additional term $R_{h}$. (see Section \ref{sub:conjugated}). Then we estimate the cross-inner product between these operators (see Section \ref{sub:croos-product}), and as a final stage, we return to the original variable.

\subsection{Conjugated operator}\label{sub:conjugated}
For all $i=1,\ldots,n$, let us consider the functions $\gamma_{i}$ such that $\mbox{reg}(\gamma)<c_0$ and the following notation
\begin{equation*}
   \nabla_{\gamma}f:=(\sqrt{\gamma_1}D_1 f_1,\cdots,\sqrt{\gamma_n}D_nf_n)\quad\text{and}\,\quad \Delta_{\gamma}f:=\sum_{i=1}^{n}\gamma_i \partial_{x_i}^{2}f. 
\end{equation*}
\par Let $\displaystyle \tilde{\mathcal{P}}(w):=\,dw+\sum_{i=1}^{n}D_{i}\left(\gamma_iD_{i}(w)\right)dt=fdt+gdB(t)$. Building on the strategy of \cite{LPP:2025}, we have the following identity in $Q$
\begin{equation}\label{eq:opconjugate}
    r\mathcal{P}(\rho z)+M_h (z)dt=C(z)\,dt+B(z),
\end{equation}
where $C(z):= C_{1}(z)+C_{2}(z)+C_3(z)+C_4(z)+C_5(z)$, $B(z):= B_1(z)+(B_{2}(z)+B_3(z))
\,dt$ and \break $M_{h}(z):= C_4(z)+C_5(z)+B_3(z)-R_hz$. The definitions of the $C_i(z)$, $B_i(z)$ and $R_{h}z$ are given by: $\displaystyle C_{1}(z):= \sum_{i=1}^{n}rA_{i}^{2}\rho\,D_{i}(\gamma_{i}D_{i}z)$, $\displaystyle C_{2}(z):= \sum_{i=1}^{n}\gamma_{i}rD_{i}^{2}\rho A_{i}^{2}z$, $C_{3}(z):= r\partial_t(\rho) z$, \break $\displaystyle B_{1}(z):= dz$, $B_2(z):=\,2\sum_{i=1}^{n}\gamma_{i}rD_{i}A_{i}\rho D_{i}A_{i}z$,
\begin{equation*}
    \begin{split}
        R_{h}(z):=\,&\sum_{i=1}^{n}\left(h\mathcal{O}(1)rD_{i}^{2}\rho +D_{i}\gamma_{i}rA_{i}D_{i}\rho\right)A_{i}^{2}z+\sum_{i=1}^{n}h\mathcal{O}(1)rD_{i}A_{i}\rho D_{i}A_{i}z\\
        &+\sum_{i=1}^{n}\frac{h^{2}}{4}D_{i}\gamma_{i}rD_{i}^{2}\rho D_{i}A_{i}z+\sum_{i=1}^{n}\frac{h^{2}}{4}D_{i}\gamma_i \,rD_{i}A_{i}\rho D_{i}^{2}z.
    \end{split} 
\end{equation*}
and the adding terms $\displaystyle C_{4}(z):=\,\frac{h^2}{4}\sum_{i=1}^{n}D_i(\gamma_i D_i(rD_i^2\rho)A_iz)$, $\displaystyle C_5(z):=\,\frac{h^2}{4}\sum_{i=1}^{n}D_i(D_i(\gamma_irD^2_i\rho)A_iz)$, and $B_3(z):=\,-2s(\Delta_{\gamma}\varphi)\, z$.
Moreover, we can obtain the following identity
\begin{equation}\label{eq:Estirf}
    \E \int_0^T |rf|^2\,dt+\E \int_{Q}|M_{h}(z)|^2\,dt \geq 2\E\int_{Q} C(z)\,B(z).
\end{equation}
The next step is to provide an estimate for the right-hand side of \eqref{eq:Estirf}.
\begin{equation}\label{conmu}
    2\E \int_Q C(z)B(z) =2\E \sum_{i=1}^{5}\sum_{j=1}^{3}\int_Q C_{i}(z)B_{j}(z):=\,\sum_{i=1}^{5}\sum_{j=1}^{3}I_{ij}. 
\end{equation}

\subsection{An estimate for the cross-product}\label{sub:croos-product}
To obtain an estimate of the cross-product, our strategy follows \cite{LPP:2025}, where the terms are classified into three groups: those involving the differential $dz$, those involving additional terms and those involving the differential $dt$. For each case, we derive the corresponding result, which will be presented in the following. For the reader's convenience, the respective proofs are provided in Appendix \ref{missingterms} or omitted when the modifications with respect to \cite{LPP:2025} are not substantial.

Compared to \cite{LPP:2025}, the weight function in this article differs only in its temporal component, which in particular affects the estimate of the cross-product in terms of the differential $dz$ in the final expression. Analyzing this modification, we obtain the following estimate for the leading terms multiplied by $B_1(z)$, whose proof is given in Appendix \ref{missingterms}.
\begin{lemma}\label{lem:dz}(\textit{Terms that involve differential $dz$.}) For $\lambda\geq \max\{\lambda_{1},\lambda_2\}> 1$ and $\tau h (\max_{[0,T]}{\theta})\leq 1$, we have
\begin{align*}
    &\sum_{i=1}^{5}I_{i1}\geq C\E \int_{\mathcal{M}}\tau^2\lambda^3e^{2\lambda(6m+1)}|z(0)|^2+(\delta T)^{-m}\E\int_{\mathcal{M}}\tau^2\lambda^2\xi^2|\nabla_{\gamma}\psi|^2|z(T)|^2\\
    &-\E \int_{Q}s^2\lambda^2\xi^2|\nabla_{\gamma}\psi|^2|dz|^2+\sum_{i=1}^{n} \E\int_{Q_i^{\ast}}\gamma_i|D_i(dz)|^2-\left.\E\int_{\mathcal{M}_{i}^{\ast}}\gamma_i|D_iz|^2\right|_0^T-X_{1}-Y_{1},
\end{align*}
where
\begin{align*}
    &X_1:=\E\int_{Q} [s^3\lambda^2\xi^3|\nabla_\gamma\psi|^2+2s^2\mathcal{O}_{\lambda}(1)+s^3\mathcal{O}_{\lambda}((sh)^2)-\mathcal{O}_{\lambda}((sh)^2)|z|^2dt\\
    &+\E\int_{Q}\mathcal{O}_{\lambda}((sh)^2)|dz|^2-\E\int_{Q}s^2\mathcal{O}_{\lambda}((sh)^2)|z|^2dt-\E\int_{Q}\mathcal{O}_{\lambda}((sh)^2)\,|dz|^2\\
    &+\E\int_{Q}s^3\mathcal{O}_{\lambda}(1)|z|^2\,dt+\E\int_{Q}s^2\xi \varphi|dz|^2\\
&+\sum_{i=1}^{n}\left(\E\int_{Q_{i}^{\ast}}\mathcal{O}_{\lambda}((sh)^2)|D_i(dz)|^2+\E\int_{Q_i^{\ast}}s^2\mathcal{O}_{\lambda}((sh)^2)|D_iz|^2dt\right)
\end{align*}
and
\begin{equation*}
\begin{split}
    Y_{1}:=\,&\sum_{i=1}^{n}\E\left. \int_{\mathcal{M}_i^*}\mathcal{O}_{\lambda}((sh)^2)\,|D_iz|^2\right|_0^T+\E\int_{\mathcal{M}}\left[\tau^2\lambda^2\xi^2|\nabla_{\gamma}\psi|^2+\mathcal{O}_{\lambda}((sh)^2)\right]\,|z(0)|^2\\
    &-\E\int_{\mathcal{M}}\tau \lambda e^{6\lambda(m+1)}|z(T)|^2.
\end{split}
\end{equation*}
\end{lemma}
Our next step is to derive an estimate for the so-called \emph{correction terms} from \eqref{conmu}. In this case, the modifications arise in the terms multiplied by $C_3(z)$, since they involve the temporal derivative $\partial_t\rho$. Therefore, the term that requires analysis is $I_{33}$, while the estimates for the remaining terms follow the same arguments as in \cite{LPP:2025}. We thus obtain the following result, with the detailed analysis of the estimate for $I_{33}$ provided in the Appendix.

\begin{lemma}\label{lem:AdditionalsTerms}
    (\textit{product of the additional terms.}) For $\lambda\geq \lambda_1\geq 1$ and \\ $\tau h (\max_{[0,T]}{\theta})\leq 1$, we obtain
    \begin{align*}
        \sum_{i=4}^{5}I_{i2}+\sum_{i=1}^{5}I_{i3}\geq\, \sum_{i=1}^{n}\E\int_{Q_{i}^{\ast}}4s\lambda^2\xi\gamma_i&|\nabla_{\gamma}\psi|^2\,|D_iz|^2\,dt\\
        &-\E\int_{Q}4s^3\lambda^4\xi^3|\nabla_{\gamma}\psi|^4\,|z|^2\,dt-X_2
    \end{align*}
where
\begin{equation*}
    \begin{split}
        &X_2:=\, \sum_{i=1}^{n} \E\int_{Q_i^{\ast}}s|\mathcal{O}_{\lambda}((sh)^2)|\,|D_iz|^2\,dt+\sum_{i=1}^{n} \E\int_{Q_i^{\ast}}s\mathcal{O}_{\lambda}(sh)\,|D_iz|^2\,dt \\
        &+\sum_{i=1}^{n}\E\int_{Q_i^{\ast}}(s\lambda\mathcal{O}(1)+\mathcal{O}_{\lambda}(sh)\,|D_iz|^2\,dt+\E\int_{Q} s|\mathcal{O}_{\lambda}((sh)^2)|\,|z|^2\,dt\\
        &+\E\int_{Q}s^2\mathcal{O}_{\lambda}(1)\,|z|^2\,dt+\E\int_{Q}(s^3\lambda^3\xi^3\mathcal{O}(1)+s^2\mathcal{O}_{\lambda}(1)+s^3\mathcal{O}_{\lambda}(sh))\,|z|^2\,dt\\
        &+\E\int_{Q}s^3|\lambda^2\xi^2+\lambda\varphi\xi^2-\varphi\xi^2||\mathcal{O}(1)|\,|z|^2\,dt-\E\int_{Q}s|\mathcal{O}_{\lambda}((sh)^2)|\,|z|^2\,dt.
    \end{split}
\end{equation*}
\end{lemma}
Finally, the terms $I_{12}$ and $I_{22}$ are similar to those in the deterministic case discussed in \cite{BLR:2014}, since the temporal variable does not play a significant role. For this reason, we do not provide a detailed proof of their estimation. However, as in the previous case, the term $I_{32}$ involves $C_3(z)$, which depends on the temporal derivative. Therefore, by combining the estimates of $I_{12}$ and $I_{22}$ with that of $I_{32}$, we obtain the following result, whose estimate of $I_{32}$ is presented in the Appendix.
\begin{lemma}\label{lem:dt}
(\textit{Terms involving the differential $dt$.}) For $\tau h (\max_{[0,T]}{\theta})\leq 1$, we obtain
\begin{equation*} 
    \sum_{i=1}^{3}I_{i2}\geq\,\E\int_{Q} 6s^3\lambda^4\xi^3|\nabla_{\gamma}\psi|^4\,|z|^2\,dt-\sum_{i=1}^{n}\E\int_{Q_{i}^{\ast}}2s\lambda^2\xi \gamma_{i}|\nabla_{\gamma}\psi|^2\,|D_iz|^2\,dt-X_3-Y_2
\end{equation*}
where
\begin{equation*}
    \begin{split}
        &X_{3}:=\,\sum_{i=1}^{n}\E\int_{Q}|s\lambda\xi\mathcal{O}(1)+\mathcal{O}_{\lambda}(sh)+s\mathcal{O}_{\lambda}(sh)+s\mathcal{O}_{\lambda}((sh)^{2})|\,|D_{i}A_{i}z|^{2}\,dt\\
    &+\sum_{\substack{i,j=1\\ i\ne j}}^{n}\E\int_{Q_{ij}^{*}}\left|h\lambda\mathcal{O}(sh)+h\mathcal{O}_{\lambda}((sh)^{2})\right|\,|D_{ij}^{2}z|^{2}\,dt\\
    &+\sum_{i=1}^{n}\E\int_{Q^{\ast}_{i}}|s\lambda\xi\mathcal{O}(1)+s\mathcal{O}_{\lambda}(sh)+h\mathcal{O}_{\lambda}(sh)+s\mathcal{O}_{\lambda}((sh)^{2})+\mathcal{O}_{\lambda}((sh)^2)|\, |D_{i}z|^{2}\\
    &+\E\int_{Q}(s^2\lambda^3\xi^2\mathcal{O}(1)+s^2\mathcal{O}_{\lambda}(1)+s^2\mathcal{O}_{\lambda}(1)+s^3\mathcal{O}_{\lambda}((sh)^2))|z|^2\\
    &+\sum_{i=1}^{n}\E\int_{Q}h^{2}|\mathcal{O}_{\lambda}(sh)|\,|D_{i}^{2}z|^{2}\,dt+\sum_{i=1}^{n}\E\int_{Q}\left|h\lambda\mathcal{O}(sh)+h\mathcal{O}_{\lambda}((sh)^{2})\right|\,|D_{i}^{2}z|^{2}\,dt
    \end{split}
\end{equation*}
and
\begin{equation*}
    \begin{split}
        Y_{2}&:=\,\sum_{i=1}^{n}\E\int_{\partial_{i}Q}\left(-2s\lambda\xi(\gamma_i)^2\partial_{i}\psi+s\mathcal{O}_{\lambda}(sh)+h\mathcal{O}_{\lambda}(sh)\right)t_{r}^{i}(|D_{i}z|^{2})\nu_{i}\,dt\\
        &+\sum_{i=1}^{n}\E\int_{\partial_iQ}\mathcal{O}_{\lambda}(sh)t_r^{i}(|D_iz|^2)\,dt+\sum_{i,j=1}^{n}E\int_{\partial_{i}Q}s\mathcal{O}_{\lambda}((sh)^2)\,t_{r}^{i}(|D_{i}z|^{2})\nu_{i}\,dt.
    \end{split}
\end{equation*}
\end{lemma}
Then by Lemma \ref{lem:dz}-\ref{lem:dt}, from \eqref{conmu} we obtain, for $\lambda\geq\max\{\lambda_1,\lambda_2\}$ and $0<\tau h (\max_{[0,T]}{\theta})\leq \varepsilon_1(\lambda)$
\begin{equation}\label{eq:sumterms}
\begin{split}
    2\E \int_Q C(z)B(z) \geq&\,-\sum_{j=1}^{3}X_{j}-\sum_{j=1}^2 Y_j+\E\int_{Q}2s^{3}\lambda^{4}\xi^{3}|\nabla_{\gamma}\psi|^{4}\,|z|^{2}\,dt\\
    &+ \sum_{i=1}^{n}\E\int_{Q_{i}^{\ast}}2s\lambda^{2}\xi\gamma_i|\nabla_{\gamma}\psi|^{2}\,|D_{i}z|^{2}\,dt-\E\int_{Q}s^2\lambda^2\xi^2|\nabla_{\gamma}\psi|^2|dz|^2\\
    &+\sum_{i=1}^{n} \E\int_{Q^{\ast}}\gamma_i |D_i(dz)|^2+C\E\int_{\mathcal{M}}\tau^2\lambda^3e^{2\lambda(6m+1)}|z(0)|^2\\
    &+(\delta T)^{-m}\E\int_{\mathcal{M}}\tau^2\lambda^2\xi^2|\nabla_{\gamma}\psi|^2|z(T)|^2-\sum_{i=1}^{n}\left.\E\int_{\mathcal{M}_{i}^{\ast}}\gamma_i|D_iz|^2\right    |_0^T.
    \end{split}
\end{equation}
To give an estimate of the right-hand side of \eqref{eq:Estirf}, we need the following estimation of $M_hz$, see \eqref{eq:opconjugate}. Similarly to \cite{LPP:2025}, the proof can be adapted from Lemma 4.2 in \cite{BHLR:2010a} and the estimation of $\Phi$ in \cite{zhao:2024}.
\begin{lemma}\label{lem:estimate:termsMh}
(\textit{Estimate of $M_h(z)$.)} For $\tau h (\max_{[0,T]}{\theta})\leq 1$, we have
\begin{equation*}
    \E\int_{Q}|M_h(z)|^2\,dt\leq \mathcal{O}_{\lambda}(1)\left(\E\int_{Q}s^2|z|^2\,dt+h^2\sum_{i=1}^{n}\int_{Q^{\ast}}s^2|D_iz|^2\,dt\right).
\end{equation*}
\end{lemma}
Combining the above lemma with \eqref{eq:sumterms}, we see that if we choose $\lambda_{0}\geq \max\{\lambda_{1},\lambda_2\}$ sufficiently large, then $\lambda=\lambda_0$(fixed for the rest of the proof), $0<\tau h(\max_{[0,T]}\theta)\leq \varepsilon_{1}(\lambda_0)$ and $0<h\leq h_1(\lambda_0)$, we have
\begin{equation}\label{eq:conmutador}
\begin{split}
\E\int_{Q}|rf|^2\,dt&+\E\int_{Q}s^2\lambda^2\xi^2|\nabla_{\gamma}\psi|^2\,|dz|^2+\E\int_{\mathcal{M}}\gamma_i|D_iz(T)|^2\\
   \geq& \,\E\int_{Q}2s^{3}\lambda^{4}\xi^{3}|\nabla_{\gamma}\psi|^{4}\,|z|^{2}\,dt+ \sum_{i=1}^{n}\E\int_{Q_{i}^{\ast}}s\lambda^{2}\xi\gamma_i|\nabla_{\gamma}\psi|^{2}\,|D_{i}z|^{2}\,dt\\
   &+\E\int_{\mathcal{M}}\tau^2\lambda^3e^{2\lambda(6m+1)}|z(0)|^2+\sum_{i=1}^{n}\E\int_{\mathcal{M}_{i}^{\ast}}\gamma_i|D_iz(0)|^2\\
   &+(\delta T)^{-m}\E\int_{\mathcal{M}}\tau^2\lambda^2\xi^2|\nabla_{\gamma}\psi|^2\,|z(T)|^2+\tilde{X}+\tilde{Y},
   \end{split}
\end{equation}
with
\begin{equation}
\begin{split}
    \tilde{X}:=&\sum_{l=1}^{3}X_{l}+\sum_{m=1}^{3}Y_{m}-\E\int_{Q}s^{3}\lambda^{3}\xi^{3}\mathcal{O}(1)|z|^{2}\,dt.
    \end{split}
\end{equation}
We can now choose $\varepsilon_{0}$ and $h_{0}$ sufficiently small, with $0<\varepsilon_{0} \leq \varepsilon_{1}(\lambda_{0})$, $0<h_0\leq h_1(\lambda_0)$, and $\tau_{1}\geq 1$ sufficiently large, such that for $\tau\geq \tau_1$ (meaning, in particular, that $s(t)$ is taken sufficiently large), $0<h\leq h_{0}$, and $\tau h(\max_{[0,T]}\theta)\leq \varepsilon_{0}$, we obtain 
\begin{equation}\label{eq:firstestimate}
\begin{split}
C_{s_0,\varepsilon}&\left(\E\int_{Q}|rf|^2\,dt+\E\int_{Q}s^2\lambda_{0}^2\xi^2|\nabla_{\gamma}\psi|^2\,|dz|^2+\E\int_{\mathcal{M}}\gamma_i|D_iz(T)|^2\right)\\
   &\geq \,\E\int_{Q}2s^{3}\lambda_{0}^{4}\xi^{3}|\nabla_{\gamma}\psi|^{4}\,|z|^{2}\,dt+ \sum_{i=1}^{n}\E\int_{Q_{i}^{\ast}}s\lambda_{0}^{2}\xi\gamma_i|\nabla_{\gamma}\psi|^{2}\,|D_{i}z|^{2}\,dt\\
   &+\E\int_{\mathcal{M}}\tau^2\lambda_{0}^3e^{2\lambda(6m+1)}|z(0)|^2+\sum_{i=1}^{n}\E\int_{\mathcal{M}_{i}^{\ast}}\gamma_i|D_iz(0)|^2+\tilde{Y},
    \end{split}
\end{equation}
where
\begin{equation}
    \begin{split}
        \tilde{Y}:= \,\sum_{i=1}^{n}\E\int_{\partial_i Q }2s\lambda \xi &(\gamma_i)^2\partial_i\psi\, t_r^i((D_iz|^2)\nu_i\,dt.
    \end{split}
\end{equation}
Moreover, since $|D_iz|^2\leq\,Ch^{-2}(|\tau_{-i}v|^2+|\tau_{+i}v|^2)$ and remember that $\psi$ satisfies \eqref{assumtion:psi}, we can conclude the following 
\begin{equation}\label{eq:finalestimateinz}
\begin{split}
&\E\int_{Q}s^{3}\lambda_{0}^{4}\xi^{3}\,|z|^{2}\,dt+ \sum_{i=1}^{n}\E\int_{Q_{i}^{\ast}}s\lambda_{0}^{2}\xi\,|D_{i}z|^{2}\,dt+\E\int_{\mathcal{M}}\tau^2\lambda_{0}^3e^{2\lambda(6m+1)}|z(0)|^2\\
&+\sum_{i=1}^{n}\E\int_{\mathcal{M}_{i}^{\ast}}|D_iz(0)|^2\leq\,C_{s_0,\varepsilon}\bigg(\E\int_{Q}|rf|^2\,dt+\E\int_{Q}s^2\lambda_{0}^2\xi^2\,|dz|^2+\frac{1}{h^2}\E\int_{\mathcal{M}}|z(T)|^2\bigg)\\
&+\E\int_{0}^T\int_{G_{1}\cap \mathcal{M}}s^{3}\lambda^{4}_{0}\xi^{3}\,|z|^{2}\,dt+ \sum_{i=1}^{n}\E\int_0^T\int_{G_{1}\cap\mathcal{M}_{i}^{\ast}}s\lambda^{2}_{0}\xi\,|D_{i}z|^{2}\,dt.
\end{split}
\end{equation}
\subsection{Return to original variable} Finally, we return to our original function. Similarly to \cite{LPP:2025}, we can obtain the following result using the same argument mentioned in the final part of the proof of Theorem 1.4 in \cite{BLR:2014}. 
\begin{lemma}\label{lem:inequalityDz-Dw}
For $\tau h (\max_{[0,T]}{\theta})\leq 1$, we have for each $i=1,...,n$ 
\begin{equation*}
   \begin{split}
        \E\int_{Q_{i}^{\ast}}s\xi \lambda^2&|rD_i w|^2\,dt\leq C\left(\E\int_{Q_i^{\ast}}s\xi\lambda^2\,|D_iz|^2\,dt+\E\int_{Q_{i}^{\ast}}s\mathcal{O}_{\lambda}((sh)^2)\,|D_iz|^2\,dt\right.\\
        &\left. \E\int_{Q}s^3\xi^3\lambda^4|z|^2\,dt+\E\int_{Q}s\mathcal{O}_{\lambda}((sh)^2)\,|z|^2\,dt+\E\int_{Q}s^3\mathcal{O}_{\lambda}((sh)^2)\,|z|^2\,dt\right),
    \end{split}
\end{equation*}
\begin{align*}
    &\E\int_{\mathcal{M}_{i}^{\ast}}|rD_iw(0)|^2\leq C\left(\E\int_{\mathcal{M}_{i}^{\ast}}e^{4\tau \varphi(x)}|D_iz(0)|^2+\E\int_{\mathcal{M}_{i}^{\ast}}\mathcal{O}_{\lambda}((sh)^2)|D_iz(0)|^2\right.\\
    &\left.+\E\int_{\mathcal{M}}\tau^2\lambda^2e^{2\lambda(6m+1)}e^{4\tau\varphi(x)}|z(0)|^2+\E\int_{\mathcal{M}}\mathcal{O}_{\lambda}((sh)^2|z(0)|^2\right)
\end{align*}
and
\begin{equation*}
    \begin{split}
        \E\int_0^T&\int_{G_{1}\cap\mathcal{M}_i^{\ast}}s\xi\lambda^2|D_iz|^2\,dt\\
        &\leq\, C\left(\E\int_0^T\int_{\mathcal{O}_{1}\cap\mathcal{M}_i^{\ast}}s\xi\lambda^2\,|rD_iw|^2\,dt+\E\int_0^T\int_{\mathcal{O}_{1}\cap \mathcal{M}_{i}^{\ast}}s\mathcal{O}_{\lambda}((sh)^2)\,|D_iz|^2\,dt\right.\\
        &\left. \E\int_0^T\int_{\mathcal{O}_{1}\cap\mathcal{M}}s^3\varphi^3\lambda^4|z|^2\,dt+\E\int_0^T\int_{\mathcal{O}_{1}\cap\mathcal{M}}s^3\mathcal{O}_{\lambda}((sh)^2)\,|z|^2\,dt\right).
    \end{split}
\end{equation*}
\end{lemma}

Moreover, Combining the above lemma with \eqref{eq:finalestimateinz} and noting
\begin{equation}\label{eq:equaldw^2}
    \E\int_{Q}s^2|dz|^2=\E\int_{Q}s^2|rdw|^2=\E\int_{Q}s^2|rg|^2\,dt,
\end{equation}
we conclude the following lemma.
\begin{lemma}\label{lem:last}
    Given any $\lambda>\lambda_0$, exits $\varepsilon_{0}$ and $h_{0}$ sufficiently small, with $0<\varepsilon \leq \varepsilon_0(\lambda_{0})$, and $\tau_{2}\geq \tau_{1}$ sufficiently large, such that for $\tau\geq \tau_{2}$, $0<h\leq h_{0}$, and $\tau h(\max_{[0,T]}\theta)\leq \varepsilon_{0}$, 
\begin{equation*}
\begin{split}  
&\E\int_{Q}s^{3}\lambda_{0}^{4}\xi^{3}e^{2s\varphi}\,|w|^{2}\,dt+ \sum_{i=1}^{n}\E\int_{Q_{i}^{\ast}}s\lambda_{0}^{2}\xi e^{2s\varphi}\,|D_{i}w|^{2}\,dt\\
&+\E\int_{\mathcal{M}}\tau^2\lambda_{0}^3e^{2\lambda_{0}(6m+1)}e^{4\tau\varphi}|w(0)|^2+\sum_{i=1}^{n}\E\int_{\mathcal{M}_{i}^{\ast}}e^{4\tau \varphi(x)}|D_iw(0)|^2\\
&\leq\,\E\int_{0}^T\int_{G_{1}\cap \mathcal{M}}s^{3}\lambda^{4}_{0}\xi^{3}e^{2s\varphi}\,|w|^{2}\,dt+ \sum_{i=1}^{n}\E\int_0^T\int_{G_{1}\cap\mathcal{M}_{i}^{\ast}}s\lambda^{2}_{0}\xi e^{2s\varphi}\,|D_{i}w|^{2}\,dt\\
&+C_{s_0,\varepsilon}\left(\E\int_{Q}e^{2s\varphi}\,|f|^2\,dt+\E\int_{Q}s^2\lambda_{0}^2\xi^2e^{2s\varphi}\,|g|^2\,dt+\frac{1}{h^2}\E\int_{\mathcal{M}}e^{2s(T)\varphi}\,|w(T)|^2\right).
\end{split}
\end{equation*}
\end{lemma}
Now, we will need an estimate for the second term on the right-hand side in the previous lemma. For this purpose, we obtain the following Lemma.
\begin{lemma}\label{lem:estimate:termDw}
For $0<\tau h (\max_{[0,T]}\theta)\leq 1$, we have
\begin{align*}
    \sum_{i=1}^{n}&\E\int_0^T\int_{\mathcal{M}_{i}^{\ast}\cap G_{1}}s\lambda^{2}_{0}\xi e^{2s\varphi}\,|D_{i}w|^{2}\,dt\leq\sum_{i=1}^{n}\E\int_{0}^{T}\int_{\mathcal{M}_{i}^\ast\cap G_{1}}s\lambda_0^2\xi e^{2s\varphi}|D_iw|^2\,dt\\
    &\leq C_{s_3,\varepsilon_1}\left(\frac{1}{h}\E\int_{\mathcal{M}_0}\xi e^{2s(T)\varphi}|w(T)|^2+\E\int_0^T\int_{\mathcal{M}_0}s^3\lambda^2\xi^3 e^{2s\varphi}|w|^2\,dt\right.\\
   &+\E\int_{0}^{T}\int_{\mathcal{M}_0}s^2\mathcal{O}_{\lambda}(1)e^{2s\varphi}\,|w|^2\,dt+\E\int_0^T\int_{\mathcal{M}_0}s^2\lambda ^2\xi^2e^{2s \varphi}|w|^2\,dt\\
   &\left.+\E\int_0^T\int_{\mathcal{M}_0}\lambda^{-2}e^{2s \varphi}\,|f|^2\,dt\right)
\end{align*}
\end{lemma}
For a proof, see Appendix \ref{sec:intermediateresult}.

By Lemma \ref{lem:last}-\ref{lem:estimate:termDw} and observing that since $\max_{[0,T]}\theta \leq \frac{1}{\delta T^2(1+\delta)}\leq \frac{1}{T^2 \delta}$, a sufficient condition for $\tau h(\max_{[0,T]}\theta) \leq \varepsilon_{0}$ become $\tau h(T^2\delta)^{-1}\leq \varepsilon_{0}$, we complete the proof of Theorem \ref{theo:Carleman}.
\end{proof}
\section{Technical steps to obtain the estimate for the missing terms of the cross-product}\label{missingterms}

\subsection{Estimate of \texorpdfstring{$I_{11}$.}{}}
We set $\beta_{11}=rD_i^2\rho$. Then,  using $rA_i^2(\rho)=r\rho+\frac{h^2}{2}D^2_i(\rho)$ and noticing that $r\rho=1$, it follows that
\begin{equation*}
\begin{split}
I_{11}=2\E \int_{Q} C_1(z)B_1(z)=& 2\sum_{i=1}^{n}\left(\E \int_{Q} D_i(\gamma_iD_iz)\,dz+\frac{h^2}{4}\E\int_{Q}\beta_{11}\, D_i(\gamma_iD_iz)\,dz\right).
\end{split}
\end{equation*}
Thus, using the definition of $C_4z$ and $B_1z$ on the last integral above, we can rewrite  $I_{11}$ as
\begin{equation}\label{eq:int:I11}
    \begin{split}
  I_{11}=&\sum_{i=1}^{n}\left(-\E\left. \int_{\mathcal{M}_i^*}\gamma_{i}|D_iz|^2\right|_0^T+\E\int_{Q^{\ast}_{i}}\gamma_{i}|D_i(dz)|^2-\frac{h^2}{4}\left[ \left.\E\int_{\mathcal{M}^{\ast}_i}\gamma_i A_i(\beta_{11})|D_iz|^2\right|_0^T \right.\right. \\
  &-\E\int_{Q_{i}^{\ast}}\gamma_i A_i(\beta_{11})|D_i(dz)|^2-\E\int_{Q_{i}^{\ast}}\partial_t(\gamma_i A_i(\beta_{11}))|D_iz|^2dt\\
  &\left.\left.-\left.\E \int_{\mathcal{M}}D_i(\gamma_iD_i(\beta_{11}))|z|^2\right|_0^T+\E \int_{Q}\partial_t(D_i(\gamma_iD_i(\beta_{11})))|z|^2dt\right.\right.\\
  &\left.\left.+\E\int_{Q}D_i[\gamma_iD_i(\beta_{11})]|dz|^2\right]\right)-2\E\int_{Q} C_4z B_1z. 
\end{split}
\end{equation}
Now, we have to estimate the terms with $\beta_{11}=rD_i^2\rho$ on the previous expression, using \cite[Theorem 3.5]{AA:perez:2024} and \cite[Lemma 3.1]{AA:perez:2024} we obtain
\begin{equation}
    \begin{aligned}
     &A_i(\beta_{11})=s^2\xi^2\lambda^2(\partial_{i}\psi)^{2}+2s\mathcal{O}_{\lambda}(1)+s^2\mathcal{O}_{\lambda}((sh)^2)=s^2\mathcal{O}_{\lambda}(1),  \\
     &D_i(\beta_{11})=2s^2\xi^2\lambda^3(\partial_i\psi)^3+s^2\xi^2\lambda^2\mathcal{O}(1)+2s\mathcal{O}_{\lambda}(1)+s^2\mathcal{O}_{\lambda}((sh)^2)=s^2\mathcal{O}_{\lambda}(1),\\
     &\partial_t(\gamma_iA_i(\beta_{11}))=\frac{\theta_t}{\theta}\left(s^2\xi^2\lambda^2(\partial_{i}\psi)^{2}+2s\mathcal{O}_{\lambda}(1)+s^2\mathcal{O}_{\lambda}((sh)^2)\right)=\frac{\theta_t}{\theta}s^2\mathcal{O}_{\lambda}(1)
    \end{aligned}
\end{equation}
where in the last estimate we have used that $\gamma_i$ is time independent. Moreover, noting that $D_i(\gamma_i)=\mathcal{O}(1)$, $A_i(\gamma_i)=\mathcal{O}(1)$ and $D_i(\gamma_iD_i(\beta_{11})=D_i(\gamma_i)A_iD_i(\beta_{11})+A_i(\gamma_i)D^2_i(\beta_{11})$, we then know thanks to Theorem \cite[Theorem 3.5]{AA:perez:2024}
\begin{equation}\label{funtiongammaimport}
\begin{aligned}
&D_i(\gamma_iD_i(\beta_{11}))=s^2\mathcal{O}_{\lambda}(1),\quad\text{and}\quad\partial_t(D_i(\gamma_iD_i(\beta_{11})))=\frac{\theta_t}{\theta}s^2\mathcal{O}_\lambda(1).
\end{aligned}
\end{equation}
Combining these estimates, \eqref{eq:int:I11} and recalling the definition of $I_{41}$, we can write the following equality
\begin{equation*}
    I_{11}+I_{41}=\sum_{i=1}^{n} \E\int_{Q_i^{\ast}}\gamma_i|D_i(dz)|^2-\left.\E\int_{\mathcal{M}_i^{\ast}}\gamma_i|D_iz|^2\right|_0^T+X_{11}-Y_{11},
\end{equation*}
where 
\begin{align*}
    X_{11}:= &\sum_{i=1}^{n}\left(\E\int_{Q_{i}^{\ast}}\mathcal{O}_{\lambda}((sh)^2)|D_i(dz)|^2+\E\int_{Q_{i}^{\ast}}\frac{\theta_t}{\theta}\mathcal{O}_{\lambda}((sh)^2)\,|D_iz|^2dt\right)\\
    &-\E\int_{Q}\frac{\theta_t}{\theta} \mathcal{O}_{\lambda}((sh)^2)|z|^2dt-\E\int_{Q}\mathcal{O}_{\lambda}((sh)^2)\,|dz|^2
\end{align*}
and
\begin{align*}
    Y_{11}:= \sum_{i=1}^{n}\E\left. \int_{\mathcal{M}_i^*}\mathcal{O}_{\lambda}((sh)^2)\,|D_iz|^2\right|_0^T+\E\left.\int_{\mathcal{M}}\mathcal{O}_{\lambda}((sh)^2)\,|z|^2\right|_0^T.
\end{align*}
Finally,  we notice that $|\theta_t(t)|\leq s\mathcal{O}_{\lambda}(1)$ in $t\in [0,T/2]$, $|\theta_{t}(t)|\leq \theta^2$ in $[T/2,T]$ and $|\theta(t)|\leq(\delta T)^{-m}$ in $[T/2,T]$, so
\begin{align*}
    X_{11}\geq &\sum_{i=1}^{n}\left(\E\int_{Q_{i}^{\ast}}\mathcal{O}_{\lambda}((sh)^2)|D_i(dz)|^2+\E\int_{Q_{i}^{\ast}}s\mathcal{O}_{\lambda}((sh)^2)\,|D_iz|^2dt\right)\\
    &-\E\int_{Q}s\mathcal{O}_{\lambda}((sh)^2)|z|^2dt-\E\int_{Q}\mathcal{O}_{\lambda}((sh)^2)\,|dz|^2.
\end{align*}
\subsection{Estimate of \texorpdfstring{$I_{21}$}{}}
We set $\beta_{21}=\gamma_irD_i^2\rho$. From the definition of $I_{21}$ and using $A_i^2(z)=z^2+\frac{h^2}{2}D_i^2z$, it follows that
\begin{equation*}
\begin{split}
I_{21}=2\E \int_{Q} C_2(z)\,B_1(z)=& 2\sum_{i=1}^{n}\left(\E \int_{Q} \beta_{21}\; zdz+\frac{h^2}{4}\E\int_{Q}\beta_{21}\, D^2_iz\,dz\right).
\end{split}
\end{equation*}
Therefore, using the definition of $C_4(z)$, $B_1(z)$ on the last integral above and thanks to $I_{21}^{(a)}$ and $I_{21}^{(a)}$, we can rewrite $I_{11}$ as
\begin{equation}
    \begin{split}
        &I_{21}=-2\E\int_{Q} C_5(z)\, B_1(z)+\sum_{i=1}^{n}\left(\left.\E\int_{\mathcal{M}}\beta_{21}  |z|^2\right|_0^T-\E\int_{Q}\partial_t(\beta_{21}   )|z|^2dt\right.\\
        &- \E\int_{Q}\beta_{21}  |dz|^2-\frac{h^2}{4}\left[ \left.\E\int_{\mathcal{M}^{\ast}_i}A_i(\beta_{21})|D_iz|^2\right|_0^T\right.\left.-\E\int_{Q_{i}^{\ast}} A_i(\beta_{21})|D_i(dz)|^2\right.\\
        &\left.-\E\int_{Q_{i}^{\ast}}\partial_t(A_i(\beta_{21}))|D_iz|^2dt-\left.\E \int_{\mathcal{M}}D^2_i(\beta_{21}))|z|^2\right|_0^T\left.+\E \int_{Q}\partial_t(D^2_i(\beta_{21}))|z|^2dt\right.\right.\\
        &\left.\left.+\E\int_{Q}D^2_i(\beta_{21})|dz|^2\right]\right).
    \end{split}
\end{equation}
The result follows using \cite[Theorem 3.5]{AA:perez:2024} and \cite[Lemma 3.1]{AA:perez:2024} in terms with $\beta_{21}=\gamma_irD_i^2\rho$. In fact, 
\begin{equation}\label{eq:estimate:I21a}
    \begin{aligned}   \beta_{21}&=\gamma_is^2\lambda^2\xi^2(\partial_i\psi)^2+2s\mathcal{O}_{\lambda}(1)+s^2\mathcal{O}_{\lambda}((sh)^2)=s^2\mathcal{O}_{\lambda}(1),\\
        \partial_{t}(\beta_{21})&=\frac{\theta_t}{\theta}(\gamma_is^2\lambda^2\xi^2(\partial_i\psi)^2+2s\mathcal{O}_{\lambda}(1)+s^2\mathcal{O}_{\lambda}((sh)^2).
    \end{aligned}
\end{equation}
Now, applying product rule with respect to average operator and, we obtain the following result
\begin{equation*}
A_i(\beta_{21})=A_i(\gamma_i)A_i(rD_i^2\rho)+\frac{h^2}{4}D_i(\gamma_i)D_i(rD_i^2\rho)\\
\end{equation*}
and by product rule with respect to differential operator we have
\begin{equation*}D_i^2(\beta_{21})=A_i^2(\gamma_i)D_i^2(rD_i^2\rho)+2A_iD_i(\gamma_i)D_iA_i(rD_i^2\rho)+D_i^2(\gamma_i)A_i(rD_i^2\rho).
\end{equation*}
Consider that $A^\alpha_i(\gamma_i)=D_i^{\alpha}(\gamma_i)=A_iD_i(\gamma_i)=\mathcal{O}(1)$ for $\alpha=1,2$, and repeated application of \cite[Theorem 3.5]{AA:perez:2024} and \cite[Lemma 3.1]{AA:perez:2024}, we obtain the following estimates
\begin{equation}\label{eq:estimate:I21b}
    \begin{aligned}
&A_i(\beta_{21})=s^2\lambda^2\xi^2\mathcal{O}(1)+2s\mathcal{O}_{\lambda}(1)+s^2\mathcal{O}_{\lambda}((sh)^2),\\
&D_i^2(\beta_{21})=s^2\mathcal{O}_{\lambda}(1),\quad\text{and}\quad\partial_t(A_i(\beta_{21}))=\partial_t(D_i^2(\beta_{21}))=\frac{\theta_t}{\theta}s^2\theta\mathcal{O}_{\lambda}(1).
    \end{aligned}
\end{equation}
Combining the definition of $I_{51}$, \eqref{eq:estimate:I21a}, and \eqref{eq:estimate:I21b}, we deduce
\begin{equation*}
    I_{21}+I_{51}=\left.\E\int_{\mathcal{M}}s^2\lambda^2\xi^2|\nabla_\gamma\psi|^2|z|^2\right|_0^T-\E\int_{\mathcal{Q}}s^2\lambda^2\xi^2|\nabla_{\gamma}\psi|^2|dz|^2+X_{21}-Y_{21},
\end{equation*}
where
\begin{equation*}
    \begin{split}
        &X_{21}\\
        &:=\; -\E\int_{Q}\frac{\theta_t}{\theta} \left[s^2\lambda^2\xi^2|\nabla_{\gamma}\psi|^2+2s\mathcal{O}_{\lambda}(1)+\mathcal{O}_{\lambda}((sh)^2)\right]|z|^2dt+\E\int_{Q}\mathcal{O}_{\lambda}((sh)^2)|dz|^2\\
        &+\sum_{i=1}^{n}\left(\E\int_{Q_{i}^{\ast}}\mathcal{O}_{\lambda}((sh)^2)|D_i(dz)|^2+\E\int_{Q_i^{\ast}}\frac{\theta_t}{\theta}\mathcal{O}_{\lambda}((sh)^2)|D_iz|^2dt\right)
    \end{split}
\end{equation*}
and
\begin{equation*}
\begin{split}
Y_{21}:=&\;\left.\E\int_{\mathcal{M}}\mathcal{O}_\lambda((sh)^2)|z|^2\right|_0^T+\left.\sum_{i=1}^{n}\E\int_{\mathcal{M}_{i}^{\ast}}\mathcal{O}_{\lambda}((sh)^2)|D_iz|^2\right|_0^T.
\end{split}
\end{equation*}
Since there exists $\lambda_1$ such that for $\lambda >\lambda_1$, we have $|\theta_{t}(t)|\leq \tau \xi$ for $t\in [0,T/2]$. In addition, $|\theta_{t}(t)|\leq \theta^2(t)$ in $[T/2,T]$ and $|\theta(t)|\leq (\delta T)^{-m}$ in $[T/2,T]$. Thus, we obtain
\begin{align*}
    &X_{21}\\
    &\geq\; -\E\int_{Q} \left[s^3\lambda^2\xi^3|\nabla_{\gamma}\psi|^2+2s^2\mathcal{O}_{\lambda}(1)+s\mathcal{O}_{\lambda}((sh)^2)\right]|z|^2dt+\E\int_{Q}\mathcal{O}_{\lambda}((sh)^2)|dz|^2\\
        &+\sum_{i=1}^{n}\left(\E\int_{Q_{i}^{\ast}}\mathcal{O}_{\lambda}((sh)^2)|D_i(dz)|^2+\E\int_{Q_i^{\ast}}s\mathcal{O}_{\lambda}((sh)^2)|D_iz|^2dt\right).
\end{align*}
\subsection{Estimate of \texorpdfstring{$I_{31}$}{}}
\begin{align*}
I_{31}=&\left.\E\int_{\mathcal{M}}r\partial_{t}(\rho)\,  |z|^2\right|_0^T-\E\int_{Q}\partial_t(r\partial_{t}(\rho))\,|z|^2dt- \E\int_{Q}r\partial_{t}(\rho)\, |dz|^2\\
=&\left.-\E \int_{\mathcal{M}}\frac{\theta_t}{\theta}s\varphi|z|^2\right|_0^T+\E\int_{Q}\frac{\theta_{tt}}{\theta}s\varphi|z|^2\,dt+\E\int_{Q}\frac{\theta_t}{\theta}s\varphi|dz|^2.
\end{align*}
From the definition of $\theta$ and $\varphi$, we can see that $\theta_t(T)=m(\delta T)^{-m-1}$ and 
$$
\tau\theta_t(0)\varphi=4T^{-1}\tau^2\lambda^2e^{\lambda(6m-4)}(\lambda e^{6\lambda (m+1)}-e^{\lambda(\psi(x)+6m)})>C\tau^2\lambda^3e^{2\lambda(6m+1) }
$$
for $\lambda>1$ and $C$ is uniform with respect to $T$. Moreover, taking into account the properties of the temporal weight function commented at the beginning of the section and noting that there exists $\lambda_2>1$ so that $\lambda\geq \lambda_2$, we have $\lambda^2e^{-2\lambda}\leq 1$, we can obtain the following result for $t\in [0,T/2)$
\begin{align*}
\frac{|\theta_t\varphi|}{\theta}s&\leq \tau\sigma \lambda e^{6(m+1)}=\tau^2\lambda^3e^{\lambda(12m+2)}=s^2\mathcal{O}_{\lambda}(1),\\
    \frac{|\theta_{tt}\varphi|}{\theta}s\leq& \tau \sigma^2\lambda e^{6\lambda(m+1)}=\tau^3\lambda^5e^{\lambda(18m-2)}\leq \tau^3\lambda^3e^{18\lambda m}\leq s^3\lambda^3\xi^3
\end{align*}
and for $t\in [T/2,T]$, we have $|\theta_t\varphi|\leq\theta^2 \mathcal{O}_{\lambda}(1) $ and $|\theta_{tt}\varphi|\leq \theta^3\mathcal{O}_{\lambda}(1)$. Therefore, 
\begin{align*}  I_{31}>&C\E\int_{\mathcal{M}}\tau^2\lambda^3e^{2\lambda(6m+1)}|z(0)|^2-\E\int_{\mathcal{M}}\tau \lambda e^{6\lambda(m+1)}|z(T)|^2\\
&-\E\int_{Q}s^3\lambda^3\xi^3\mathcal{O}(1)|z|^2\,dt-\E\int_{\mathcal{M}}s^2\xi\varphi |dz|^2.
\end{align*}
\subsection{Estimate of \texorpdfstring{$I_{32}$}{}}
Denoting $\beta^{i}_{32}:= \gamma_i r^2\partial_t(\rho)D_iA_i\rho$, we have
\begin{align*}
    2\E\int_{Q}C_{3}(z)B_{2}(z)\,dt=2\sum_{i=1}^{n}\E\int_{Q} \beta^{i}_{32}zD_iA_iz\,dt:= \sum_{i=1}^{n}I^{i}_{32}.
\end{align*}
Firstly, using integration by parts with respect to the operator differential and product rule, we obtain 
\begin{align*}
    I^{i}_{32}=&-2\E\int_{Q_i^{\ast}}D_i(\beta_{32}^{i} z)A_iz\,dt+2\E\int_{\partial_i Q}\beta_{32}^{i}z \,t_{r}^{i}(A_iz)\nu_{i}\,dt\\
    =&-2\E\int_{Q_i^{\ast}}D_i(\beta_{32}^{i}) |A_iz|^2\,dt- 2\E\int_{Q_i^{\ast}}A_i(\beta_{32}^{i})D_i(z)A_iz\,dt
\end{align*}
where we utilized $z=0$ on $\partial_i Q$. Now, using the identity $2D_i(z)A_iz = D_i(z^2)$ and applying the integration by parts with respect to the differential operator once again to the second integral in the above equation, we deduce that

\begin{align*}
    I_{32}^{i}=&-2\E\int_{Q_i^{\ast}}D_i(\beta_{32}^{i}) |A_iz|^2\,dt-\E\int_{Q_i^{\ast}}A_i(\beta_{32}^{i})D_i(|z|^2)\,dt\\
    =& -2\E\int_{Q_i^{\ast}}D_i(\beta_{32}^{i}) |A_iz|^2\,dt+\E\int_{Q}D_iA_i(\beta_{32}^{i})\,|z|^2\,dt-\E\int_{\partial_iQ}t_{r}^{i}(A_i\beta_{32}^{i})\,|z|^2\nu_i\,dt\\
    =&-2\E\int_{Q_i^{\ast}}D_i(\beta_{32}^{i}) |A_iz|^2\,dt+\E\int_{Q}D_iA_i(\beta_{32}^{i})\,|z|^2\,dt
\end{align*}
where condition $z=0$ on $\partial_i Q$ has been applied in the last line. 

On the other hand, by product rule with respect to differential operator and the results \cite{AA:perez:2024}, we obtain the following equalities:
\begin{align*}
    |D_i\beta_{32}^{i}|=|A_iD_i\beta_{32}^{i}|=\frac{|\theta_t|}{\theta}s^2|(\lambda^2\xi^2+\lambda\varphi\xi-\varphi\xi)\mathcal{O}(1)+\mathcal{O}_{\lambda}((sh)^2)|,
\end{align*}
where we consider $D_i\gamma_i=A_i\gamma_i=A_i^2\gamma_i=A_iD_i\gamma_i=\mathcal{O}(1)$. Now, considering there exist $\lambda_1$ such that $\lambda>\lambda_1$, it follows $|\theta_t(t)|\leq C\tau \xi $ in $[0,T/2]$, in addition, $|\theta_{t}|\leq\theta^2$ in $[T/2,T]$ and $|\theta|\leq (\delta T)^{-m}$, therefore,
\begin{align*}
    I_{32}:=&\sum_{i=1}^{n}I_{32}^{i}\\
    &\geq -\mathbb{E}\int_{Q_i^{\ast}}s^3|\lambda^2\xi^3+\lambda\varphi\xi^2-\varphi\xi^2||\mathcal{O}(1)||A_iz|^2\,dt-\mathbb{E}\int_{Q_i^{\ast}}s|\mathcal{O}_\lambda((sh)^2)||A_iz|^2\,dt\\
    &-\E\int_{Q}s^3|\lambda^2\xi^3+\lambda\varphi\xi^2-\varphi\xi^2||\mathcal{O}(1)||z|^2\,dt-\E\int_{Q}s|\mathcal{O}_{\lambda}((sh)^2)||z|^2\,dt
\end{align*}
Finally, using $|A_iz|^2\leq A_i|z|^2$, integration by parts with respect to average operator and condition $z=0$ on $\partial_iQ$, we have
\begin{align*}
      I_{32}:=\sum_{i=1}^{n}I_{32}^{i}\geq-\E\int_{Q}s^3|\lambda^2\xi^3+\lambda\varphi\xi^2-\varphi\xi^2||\mathcal{O}(1)||z|^2\,dt-\E\int_{Q}s|\mathcal{O}_{\lambda}((sh)^2)||z|^2\,dt.
\end{align*}
\subsection{Estimate of \texorpdfstring{$I_{33}$}{}}
Denoting $\beta_{22}:=\,-2sr\partial_t(\varphi)\Delta_{\gamma}(\varphi)$,  we have \\$\displaystyle
        2\E \int_{Q} C_{3}zB_{3}zdt=2\E\int_{Q}\beta_{22}\,|z|^2dt:=\,I_{33}$.
Consider $|\theta_t|\leq s\mathcal{O}_\lambda(1)$ in $[0,T/2]$, $\theta_t(t)\leq \theta^2(t)$ in $[T/2,T]$, $|\theta(t)|\leq (\delta T)^{-m}$ and observe that
$\Delta_{\gamma}(\varphi)=\lambda^2\xi|\nabla_{\gamma}\psi|^2+\lambda \xi \mathcal{O}(1)$, thus
\begin{equation}\label{eq:estimatefinalI33}
    I_{33}\geq \E\int_{Q}s^2\mathcal{O}_{\lambda}(1)|z|^2\,dt.
\end{equation}
\section{Technical steps to obtain the intermediate result}\label{sec:intermediateresult}

This section is dedicated to the proof of the Lemma \ref{lem:estimate:termDw}. We choose a function $\chi\in C_0^{\infty}(G_0;[0,1])$ such that $\chi\equiv 1$ in $G_1$ . By the Itô formula, we see that
\begin{equation*}
    \begin{split}
        d(\chi^2s\xi e^{2s\varphi}|w|^2)=&\chi^2\xi\partial_t(se^{2s\varphi})\,|w|^2\,dt+2\chi^2s\xi e^{2s\varphi}wdw+\chi^2s\lambda^2_0\xi e^{2s\varphi}|dw|^2\\
        =&\chi^2\xi e^{2s\varphi}\left[\frac{\theta_t}{\theta}(s+2s^2\varphi)|w|^2\,dt+2swdw+s|dw|^2\right].
    \end{split}
\end{equation*}
Taking into account that $z$ satisfies $\displaystyle dw=\left(-\sum_{i=1}^{n}D_i(\gamma_iD_iw)+f\right)\,dt+g\,dB(t)$, we have
\begin{equation}\label{eq1:AppendixC}
\begin{split}
    &\E\int_{\mathcal{M}_0}\chi^2(\delta T)^{-m}\tau \xi e^{2(\delta T)^{-m}\tau \varphi}|w(T)|^2-2\E\int_{\mathcal{M}_0}\chi^2 \tau \xi e^{4\tau\varphi}|w(0)|^2=\\
    &\E\int_0^T\int_{\mathcal{M}_0}\frac{\theta_t}{\theta}\chi^2\xi e^{2s\varphi}(s+2s^2\varphi)|w|^2\,dt+2\E\int_{0}^T\int_{\mathcal{M}_0}\chi^2\xi e^{2s\varphi}sfw\,dt\\
    &-2\sum_{i=1}^{n}\E\int_0^T\int_{\mathcal{M}_0}\chi^2\xi e^{2s\varphi}sD_i(\gamma_iD_iw)w\,dt +\E\int_{0}^T\int_{\mathcal{M}_0}\chi^2\xi e^{2s\varphi}s|g|^2\,dt,
\end{split}
\end{equation}
where we used $\E\,B(t)=0$ and $|dw|^2=g^2\,dt$. Using the integration by parts on the second integral on the right-hand side in the above equation, we can see that for each $i=1,\ldots,n$
\begin{align*}
    -2\E\int_{0}^T\int_{\mathcal{M}_0}\chi^2\xi e^{2s\varphi}sD_i(\gamma_iD_iw)&w\,dt
    =2\E\int_{0}^T\int_{\mathcal{M}_0}\chi^2\xi e^{2s\varphi}sD_i(\gamma_iD_iw)w\,dt
    \\
    &-\E\int_{0}^T\int_{\partial_i(\mathcal{M}_0)}\chi^2\xi se^{2s\varphi}t_r^{i}(\gamma_iD_iw)w\nu_i\,dt\\
    =&2\E\int_0^T\int_{(\mathcal{M}_0)_{i}^{*}}D_i(\chi^2\xi e^{2s\varphi}w)s\gamma_iD_iw\,dt,
\end{align*}
where we used $\chi=0$ on $\partial \mathcal{O}_0$. Now, by product rule, we obtain
\begin{align}
-2\E\int_{0}^T\int_{\mathcal{M}_0}\chi^2\xi e^{2s\varphi}sD_i(\gamma_i&D_iw)w\,dt=2\E\int_{0}^T\int_{(\mathcal{M}_0)_i^{*}}A_i(\chi^2\xi e^{2s\varphi})s\gamma_i|D_iw|^2\,dt\nonumber \\
&+2\E\int_{0}^T\int_{(\mathcal{M}_0)_i^{*}}D_i(\chi^2\xi e^{2s\varphi})s\gamma_iA_i(w)D_i(w)\,dt.\label{eq3:appendiC}
\end{align}
By results in \cite{AA:perez:2024}, we have $$D_i(\chi^2\xi e^{2s\theta\varphi})=\partial_i(\chi^2\xi e^{2s\varphi})+h\mathcal{O}_{\lambda}(sh)e^{2s\varphi}=s\lambda\mathcal{O}(1)\chi  \xi^2e^{2s\theta\varphi}+h\mathcal{O}_{\lambda}(sh)e^{2s\varphi},$$ then for each $i=1,...,n$ on the second integral of the above equation we have
\begin{equation}\label{eq5:lemma4.7}
    \begin{split}
        2&\E\int_0^T\int_{(\mathcal{M}_0)_{i}^{*}}s\gamma_iD_i(\chi^2\xi  e^{2s\varphi})A_iw\,D_iw\,dt\\
        =\,&\E\int_0^T\int_{(\mathcal{M}_0)_i^{\ast}}\left[s^2\lambda\mathcal{O}(1)\chi^2 \xi^2e^{2s\varphi}+\mathcal{O}_{\lambda}((sh)^2)e^{2s\theta\varphi}\right]A_iw\,D_iw\,dt\\
        \leq\, & \E\int_0^T\int_{(\mathcal{M}_0)_i^{\ast}}\chi^2s\xi e^{2s\varphi}|D_iz|^2\,dt+C\E\int_0^T\int_{(\mathcal{M}_0)_i^{\ast}}s^3\lambda^2\xi^3 e^{2s\varphi}|A_iw|^2\,dt\\
        &+\E\int_0^T\int_{(\mathcal{M}_0)_i^{\ast}}\mathcal{O}_{\lambda}((sh)^2)e^{2s\varphi}|D_iz|^2\,dt+\E\int_0^T\int_{\mathcal{M}_{i}^{\ast}\cap\mathcal{O}_{1}}\mathcal{O}_{\lambda}((sh)^2)e^{2s\varphi}|A_iw|^2\,dt
    \end{split}
\end{equation}
where Young's inequality was used. Noting that \\$A_i(\varphi ^3e^{2s\theta\varphi})=\varphi ^3e^{2s\theta\varphi}+\mathcal{O}_{\lambda}((sh)^2)e^{2s\theta\varphi}$, $A_i(e^{2s\theta\varphi})=\mathcal{O}_{\lambda}(1)e^{2s\theta\varphi}$, $|A_iw|^2\leq A_i|w|^2$, and integration by parts with respect to average operator, we have the following
\begin{equation}\label{eq6:lemma4.7}
    \begin{split}
\E\int_0^T&\int_{(\mathcal{M}_0)_i^{\ast}}s^3\lambda^2\xi^3 e^{2s\varphi}|A_iw|^2\,dt+\E\int_0^T\int_{(\mathcal{M}_0)_i^{\ast}}\mathcal{O}_{\lambda}((sh)^2)e^{2s\varphi}|A_iw|^2\,dt\\
        \leq&  \E\int_0^T\int_{\mathcal{M}_0}s^3\lambda^2\xi^3 e^{2s\varphi}|w|^2\,dt+\E\int_0^T\int_{\mathcal{M}_0}\mathcal{O}_{\lambda}((sh)^2)e^{2s\varphi}|w|^2\,dt.
    \end{split}
\end{equation}
Hence, combining  \eqref{eq5:lemma4.7} and \eqref{eq6:lemma4.7}, the last term on the left-hand side of \eqref{eq3:appendiC} can be estimated by
\begin{equation}\label{eq7:lemma4.7}
    \begin{split}
        2&\E\int_0^T\int_{(\mathcal{M}_0)_i^{\ast}}s\gamma_iD_i(\chi^2 \xi e^{2s\varphi})A_iw\,D_iw\,dt\\
        \leq\, & \E\int_0^T\int_{(\mathcal{M}_0)_i^{\ast}}\chi^2s\xi e^{2s\varphi}|D_iz|^2\,dt+C\E\int_0^T\int_{\mathcal{M}_0}s^3\lambda^2\xi^3 e^{2s\varphi}|w|^2\,dt\\
        &+\E\int_0^T\int_{(\mathcal{M}_0)_i^{\ast}}\mathcal{O}_{\lambda}((sh)^2)e^{2s\varphi}|D_iz|^2\,dt+\E\int_0^T\int_{\mathcal{M}_0}\mathcal{O}_{\lambda}((sh)^2)e^{2s\varphi}|w|^2\,dt.
    \end{split}
\end{equation}
Therefore, from the equation above with \eqref{eq3:appendiC}, we can see that
\begin{align*}
    2\E&\int_{0}^T\int_{(\mathcal{M}_0)_i^{*}}A_i(\chi^2\xi e^{2s\varphi})s\gamma_i|D_iw|^2\,dt\leq -2\E\int_{0}^T\int_{\mathcal{M}_0}\chi^2\xi e^{2s\varphi}sD_i(\gamma_iD_iw)w\,dt\\
    &+\E\int_0^T\int_{(\mathcal{M}_0)_i^{\ast}}\chi^2s\xi e^{2s\varphi}|D_iz|^2\,dt+C\E\int_0^T\int_{\mathcal{M}_0}s^3\lambda^2\xi^3 e^{2s\varphi}|w|^2\,dt\\
&+\E\int_0^T\int_{(\mathcal{M}_0)_i^{\ast}}\mathcal{O}_{\lambda}((sh)^2)e^{2s\varphi}|D_iz|^2\,dt+\E\int_0^T\int_{\mathcal{M}_0}\mathcal{O}_{\lambda}((sh)^2)e^{2s\varphi}|w|^2\,dt.
\end{align*}
Moreover, taking into account that $e^{-2s\varphi}A_i(\chi^2\xi e^{2s\varphi})=\chi^2\xi +\mathcal{O}((sh)^2)$, we have the following
\begin{align*}
    \E&\int_{0}^T\int_{(\mathcal{M}_0)_i^{*}}\chi^2\gamma_i s\xi e^{2s\varphi}|D_iw|^2\,dt\leq -2\E\int_{0}^T\int_{\mathcal{M}_0}\chi^2\xi e^{2s\varphi}sD_i(\gamma_iD_iw)w\,dt\\
    &+C\E\int_0^T\int_{\mathcal{M}_0}s^3\lambda^2\xi^3 e^{2s\varphi}|w|^2\,dt+\E\int_0^T\int_{(\mathcal{M}_0)_i^{\ast}}\mathcal{O}_{\lambda}((sh)^2)e^{2s\varphi}|D_iz|^2\,dt\\
&+\E\int_0^T\int_{\mathcal{M}_0}\mathcal{O}_{\lambda}((sh)^2)e^{2s\varphi}|w|^2\,dt.
\end{align*}
Combining the above equation with \eqref{eq1:AppendixC}, we obtain
\begin{align*}
   &\sum_{i=1}^{n}\E\int_{0}^{T}\int_{\mathcal{M}_{i}^\ast\cap  G_1}\gamma_is\xi e^{2s\varphi}|D_iw|^2\,dt\leq \E\int_{\mathcal{M}}\chi^2(\delta T)^{-m}\tau \xi e^{2(\delta T)^{-m}\tau \varphi}|w(T)|^2\\
   &-2\E\int_{\mathcal{M}}\chi^2 \tau \xi e^{4\tau\varphi}|w(0)|^2+C\E\int_0^T\int_{\mathcal{M}_0}s^3\lambda^2\xi^3 e^{2s\varphi}|w|^2\,dt\\
   &+\E\int_0^T\int_{(\mathcal{M}_0)_i^{\ast}}\mathcal{O}_{\lambda}((sh)^2)e^{2s\varphi}|D_iz|^2\,dt+\E\int_0^T\int_{\mathcal{M}_0}\mathcal{O}_{\lambda}((sh)^2)e^{2s\varphi}|w|^2\,dt\\
   &-\E\int_{Q}\frac{\theta_t}{\theta}\chi^2\xi e^{2s\varphi}(s+2s^2\varphi)\,|w|^2\,dt-2\E\int_{Q}\chi^2\xi e^{2s\varphi}sf w\,dt-E\int_{Q}\chi^2\xi e^{2s\varphi}s|g|^2\,dt
\end{align*}
Thus, by Cauchy-Schwarz on the seventh integral in the above equation, we have
\begin{align*}
   \sum_{i=1}^{n}&\E\int_{0}^{T}\int_{\mathcal{M}_{i}^\ast\cap  G_1}\gamma_is\xi e^{2s\varphi}|D_iw|^2\,dt\leq \E\int_{\mathcal{M}}\chi^2(\delta T)^{-m}\tau \xi e^{2(\delta T)^{-m}\tau \varphi}|w(T)|^2\\
   &+C\E\int_0^T\int_{\mathcal{M}_0}s^3\lambda^2\xi^3 e^{2s\varphi}|w|^2\,dt+\E\int_0^T\int_{(\mathcal{M}_0)_i^{\ast}}\mathcal{O}_{\lambda}((sh)^2)e^{2s\varphi}|D_iz|^2\,dt\\
   &+\E\int_0^T\int_{\mathcal{M}_0}\mathcal{O}_{\lambda}((sh)^2)e^{2s\varphi}|w|^2\,dt-\E\int_{Q}\frac{\theta_t}{\theta}\chi^2\xi e^{2s\varphi}(s+2s^2\varphi)\,|w|^2\,dt\\
   &+C \E\int_0^T\int_{\mathcal{M}_0}s^2\lambda ^2\xi^2e^{2s \varphi}|w|^2\,dt+C\E\int_0^T\int_{\mathcal{M}_0}\lambda^{-2}e^{2s \varphi}\,|f|^2\,dt
\end{align*}
Now, we focus on the fifth integral on the right-hand side in the above inequality. From the definition of $\theta$ and $\varphi$, we can notice that $-\theta_t\varphi\leq 0$ or $|\theta_t|\leq s\mathcal{O}_\lambda(1)$ on $[0,T/2$] and for $[T/2,T]$, $|\theta_t|\leq |\theta(t)|^2$. This allows us to conclude,
\begin{align*}
    -&\E\int_{Q}\frac{\theta_t}{\theta}\chi^2\xi e^{2s\varphi}(s+2s^2\varphi)\,|w|^2\,dt\\
    &=-\E\int_{0}^{T/2}\int_{\mathcal{M}}\frac{\theta_t}{\theta}\chi^2\xi e^{2s\varphi}(s+2s^2\varphi)\,|w|^2\,dt-\int_{T/2}^T\int_{\mathcal{M}}\frac{\theta_t}{\theta}\chi^2\xi e^{2s\varphi}(s+2s^2\varphi)\,|w|^2\,dt\\
    \leq &\E\int_{0}^{T/2}\int_{\mathcal{M}} |\theta(t)|\xi e^{2s\varphi}s^2\mathcal{O}_{\lambda}(1)\,|w|^2\,dt+2\int_{T/2}^{T}\int_{\mathcal{M}}|\theta(t)|\chi^2\xi e^{2s\varphi}s^2|\varphi|\,|w|^2\,dt\\
    &\leq \E\int_0^Ts^2\mathcal{O}_{\lambda}(1)\,|w|^2\,dt.
\end{align*}
Therefore, taking $h\tau (\delta T)^{-m}\leq \varepsilon_1$, we obtain the desired result.


\bibliographystyle{abbrv}
\bibliography{references}

@article{DT:2000,
 author = {de Teresa, Luz},
 title = {Insensitizing controls for a semilinear heat equation},
 fjournal = {Communications in Partial Differential Equations},
 journal = {Commun. Partial Differ. Equations},
 issn = {0360-5302},
 volume = {25},
 number = {1-2},
 pages = {39--72},
 year = {2000},
 language = {English},
 doi = {10.1080/03605300008821507},
 zbMATH = {1423428},
 Zbl = {0942.35028}
}

@article {ABM:2018,
    AUTHOR = {Allonsius, Damien and Boyer, Franck and Morancey, Morgan},
     TITLE = {Spectral analysis of discrete elliptic operators and
              applications in control theory},
   JOURNAL = {Numer. Math.},
  FJOURNAL = {Numerische Mathematik},
    VOLUME = {140},
      YEAR = {2018},
    NUMBER = {4},
     PAGES = {857--911},
      ISSN = {0029-599X,0945-3245},
   MRCLASS = {34L16 (35K10 65M06 93B05 93B40)},
  MRNUMBER = {3864704},
MRREVIEWER = {Can\ Zhang},
       DOI = {10.1007/s00211-018-0983-1},
       URL = {},
}

@article {BDK:2026,
    AUTHOR = {Bhandari, Kuntal and Dutta, Rajib and Kumar, Manish},
     TITLE = {Carleman estimate and controllability of a time-discrete
              coupled parabolic system},
   JOURNAL = {Math. Control Relat. Fields},
  FJOURNAL = {Mathematical Control and Related Fields},
    VOLUME = {16},
      YEAR = {2026},
     PAGES = {229--278},
      ISSN = {2156-8472,2156-8499},
   MRCLASS = {93B05 (35K51 65M06 93-08)},
  MRNUMBER = {4974918},
       DOI = {10.3934/mcrf.2025016},
       URL = {},
}

@article {AB:2020,
    AUTHOR = {Allonsius, Damien and Boyer, Franck},
     TITLE = {Boundary null-controllability of semi-discrete coupled
              parabolic systems in some multi-dimensional geometries},
   JOURNAL = {Math. Control Relat. Fields},
  FJOURNAL = {Mathematical Control and Related Fields},
    VOLUME = {10},
      YEAR = {2020},
    NUMBER = {2},
     PAGES = {217--256},
      ISSN = {2156-8472,2156-8499},
   MRCLASS = {93B05 (35K51 65M06 93-08 93C20)},
  MRNUMBER = {4097728},
MRREVIEWER = {Can\ Zhang},
       DOI = {10.3934/mcrf.2019037},
       URL = {},
}

@article {LT:2006,
    AUTHOR = {Labb\'e, St\'ephane and Tr\'elat, Emmanuel},
     TITLE = {Uniform controllability of semidiscrete approximations of
              parabolic control systems},
   JOURNAL = {Systems Control Lett.},
  FJOURNAL = {Systems \& Control Letters},
    VOLUME = {55},
      YEAR = {2006},
    NUMBER = {7},
     PAGES = {597--609},
      ISSN = {0167-6911,1872-7956},
   MRCLASS = {93B05 (35K05 93C25)},
  MRNUMBER = {2225370},
MRREVIEWER = {Pham Tran Nhu},
       DOI = {10.1016/j.sysconle.2006.01.004},
       URL = {},
}

@article{LPP:2026,
      title={Inverse random source and {C}auchy problems for semi-discrete stochastic parabolic equations in arbitrary Dimensions}, 
    journal = {\href{https://arxiv.org/abs/2509.03760}{arXiv:2509.03760}},
      author={Lecaros, Rodrigo and P\'erez, Ariel A. and Prado, Manuel F.},
      year={2026},
    eprint={2412.19892},
      archivePrefix={arXiv},
      primaryClass={},
      url={}, 
}

@article{AA:perez:2024,
      title={Asymptotic behavior of {C}arleman weight functions}, 
    journal = {\href{https://arxiv.org/abs/2412.19892}{arXiv:2412.19892}},
      author={Ariel A. P\'erez},
      year={2024},
    eprint={2412.19892},
      archivePrefix={arXiv},
      primaryClass={},
      url={}, 
}

@article {HS:LB:P-2023,
    AUTHOR = {V\'ictor Hern\'andez-Santamar\'ia and Le Balc'h, K\'evin and
              Peralta, Liliana},
     TITLE = {Global null-controllability for stochastic semilinear
              parabolic equations},
   JOURNAL = {Ann. Inst. H. Poincar\'e{} C Anal. Non Lin\'eaire},
  FJOURNAL = {Annales de l'Institut Henri Poincar\'e{} C. Analyse Non
              Lin\'eaire},
    VOLUME = {40},
      YEAR = {2023},
    NUMBER = {6},
     PAGES = {1415--1455},
      ISSN = {0294-1449,1873-1430},
   MRCLASS = {93B05 (35K55 60H15 93C20 93E03)},
  MRNUMBER = {4656420},
MRREVIEWER = {Can\ Zhang},
       DOI = {10.4171/aihpc/69},
       URL = {},
}

@article {HS2023,
    AUTHOR = {Hern\'andez-Santamar\'ia, V\'ictor},
     TITLE = {Controllability of a simplified time-discrete stabilized
              {K}uramoto-{S}ivashinsky system},
   JOURNAL = {Evol. Equ. Control Theory},
  FJOURNAL = {Evolution Equations and Control Theory},
    VOLUME = {12},
      YEAR = {2023},
    NUMBER = {2},
     PAGES = {459--501},
      ISSN = {2163-2472,2163-2480},
   MRCLASS = {93B05 (35K52 93C55)},
  MRNUMBER = {4520409},
MRREVIEWER = {V\~u{} Ngoc Ph\'at},
       DOI = {10.3934/eect.2022038},
       URL = {},
}

@article {BHSDT:2019,
    AUTHOR = {Boyer, Franck and Hern\'andez-Santamar\'ia, V\'ictor and de
              Teresa, Luz},
     TITLE = {Insensitizing controls for a semilinear parabolic equation: a
              numerical approach},
   JOURNAL = {Math. Control Relat. Fields},
  FJOURNAL = {Mathematical Control and Related Fields},
    VOLUME = {9},
      YEAR = {2019},
    NUMBER = {1},
     PAGES = {117--158},
      ISSN = {2156-8472,2156-8499},
   MRCLASS = {93C20 (35K20 35K91 65M06 93B05)},
  MRNUMBER = {3924861},
MRREVIEWER = {Ren\'e\ D\'ager},
       DOI = {10.3934/mcrf.2019007},
       URL = {},
}

@article {WZ:2025,
    AUTHOR = {Wang, Yu and Zhao, Qingmei},
     TITLE = {The {\it {$\phi $}}-null controllability for semi-discrete
              stochastic semilinear parabolic equations},
   JOURNAL = {ESAIM Control Optim. Calc. Var.},
  FJOURNAL = {ESAIM. Control, Optimisation and Calculus of Variations},
    VOLUME = {31},
      YEAR = {2025},
     PAGES = {Paper No. 98},
      ISSN = {1292-8119,1262-3377},
   MRCLASS = {93B05 (35K10 35R60 60H15 93C20 93E03)},
  MRNUMBER = {4998008},
       DOI = {10.1051/cocv/2025082},
       URL = {},
}

@article{zhao:2024,
author = {Zhao, Qingmei},
title = {Null Controllability for Stochastic Semidiscrete Parabolic Equations},
journal = {SIAM Journal on Control and Optimization},
volume = {63},
number = {3},
pages = {2007-2028},
year = {2025},
doi = {10.1137/24M1639166},
}

@article{WZ:2024,
      title={Null controllability for stochastic fourth order semi-discrete parabolic equations}, 
    journal = {\href{https://arxiv.org/abs/2405.03257}{arXiv:2405.03257}},
      author={Yu Wang and Qingmei Zhao},
      year={2024},
      archivePrefix={arXiv},
      primaryClass={math.OC},
      url={}, 
}

@article {LMPZ-2023,
    AUTHOR = {Lecaros, Rodrigo and Morales, Roberto and P\'{e}rez, Ariel and
              Zamorano, Sebasti\'{a}n},
     TITLE = {Discrete {C}arleman estimates and application to
              controllability for a fully-discrete parabolic operator with
              dynamic boundary conditions},
   JOURNAL = {J. Differential Equations},
  FJOURNAL = {Journal of Differential Equations},
    VOLUME = {365},
      YEAR = {2023},
     PAGES = {832--881},
      ISSN = {0022-0396,1090-2732},
   MRCLASS = {93B05 (65M06 93B07)},
  MRNUMBER = {4593117},
MRREVIEWER = {Salah-Eddine\ Chorfi},
       DOI = {10.1016/j.jde.2023.05.014},
       URL = {},
}

@article{CLTP-2022,
AUTHOR = {Cerpa, Eduardo and Lecaros, Rodrigo and Nguyen, Thuy N. T. and
              P\'{e}rez, Ariel},
     TITLE = {Carleman estimates and controllability for a semi-discrete
              fourth-order parabolic equation},
   JOURNAL = {J. Math. Pures Appl. (9)},
  FJOURNAL = {Journal de Math\'{e}matiques Pures et Appliqu\'{e}es. Neuvi\`eme S\'{e}rie},
    VOLUME = {164},
      YEAR = {2022},
     PAGES = {93--130},
      ISSN = {0021-7824},
   MRCLASS = {35K52 (65M06 93B05 93B07)},
  MRNUMBER = {4450878},
       DOI = {10.1016/j.matpur.2022.06.003},
       URL = {},
}

@article {LDOP-2021,
    AUTHOR = {Lecaros, Rodrigo and Ortega, Jaime H. and P\'{e}rez, Ariel and De
              Teresa, Luz},
     TITLE = {Discrete {C}alder\'{o}n problem with partial data},
   JOURNAL = {Inverse Problems},
  FJOURNAL = {Inverse Problems. An International Journal on the Theory and
              Practice of Inverse Problems, Inverse Methods and Computerized
              Inversion of Data},
    VOLUME = {39},
      YEAR = {2023},
    NUMBER = {3},
     PAGES = {Paper No. 035001, 28},
      ISSN = {0266-5611},
   MRCLASS = {42B37 (35R30 42B10 65M06)},
  MRNUMBER = {4541733},
       }

@article {BHLR:2010a,
    AUTHOR = {Boyer, Franck and Hubert, Florence and Le Rousseau, J\'{e}r\^{o}me},
     TITLE = {Discrete {C}arleman estimates for elliptic operators and
              uniform controllability of semi-discretized parabolic
              equations},
   JOURNAL = {J. Math. Pures Appl. (9)},
  FJOURNAL = {Journal de Math\'{e}matiques Pures et Appliqu\'{e}es. Neuvi\`eme S\'{e}rie},
    VOLUME = {93},
      YEAR = {2010},
    NUMBER = {3},
     PAGES = {240--276},
      ISSN = {0021-7824},
   MRCLASS = {93B05 (35K10 65M06 93C20)},
  MRNUMBER = {2601332},
       DOI = {10.1016/j.matpur.2009.11.003},
       URL = {},
}

@article {BLR:2014,
    AUTHOR = {Boyer, Franck and Le Rousseau, J\'{e}r\^{o}me},
     TITLE = {Carleman estimates for semi-discrete parabolic operators and
              application to the controllability of semi-linear
              semi-discrete parabolic equations},
   JOURNAL = {Ann. Inst. H. Poincar\'{e} Anal. Non Lin\'{e}aire},
  FJOURNAL = {Annales de l'Institut Henri Poincar\'{e}. Analyse Non Lin\'{e}aire},
    VOLUME = {31},
      YEAR = {2014},
    NUMBER = {5},
     PAGES = {1035--1078},
      ISSN = {0294-1449},
   MRCLASS = {35K58 (35B45 93B05)},
  MRNUMBER = {3258365},
MRREVIEWER = {Joseph L. Shomberg},
       DOI = {10.1016/j.anihpc.2013.07.011},
       URL = {},
}

@article {BHLR:2010b,
    AUTHOR = {Boyer, Franck and Hubert, Florence and Le Rousseau, J\'{e}r\^{o}me},
     TITLE = {Discrete {C}arleman estimates for elliptic operators in
              arbitrary dimension and applications},
   JOURNAL = {SIAM J. Control Optim.},
  FJOURNAL = {SIAM Journal on Control and Optimization},
    VOLUME = {48},
      YEAR = {2010},
    NUMBER = {8},
     PAGES = {5357--5397},
      ISSN = {0363-0129},
   MRCLASS = {35K20 (35P15 65M06 93B05 93B07)},
  MRNUMBER = {2745778},
       DOI = {10.1137/100784278},
       URL = {},
}

@article{lopez1998some,
  title={Some new results related to the null controllability of the $1-d $ heat equation},
  author={L{\'o}pez, Antonio and Zuazua, Enrique},
  journal={S{\'e}minaire {\'E}quations aux d{\'e}riv{\'e}es partielles (Polytechnique)},
  pages={1--22},
  year={1998}
}

@article{zuazua2005propagation,
  title={Propagation, observation, and control of waves approximated by finite difference methods},
  author={Zuazua, Enrique},
  journal={SIAM review},
  volume={47},
  number={2},
  pages={197--243},
  year={2005},
  publisher={SIAM}
}

@book{fursikov-1996,
  author    = {Fursikov, A. V. and Imanuvilov, O. Yu.},
  title     = {Controllability of Evolution Equations},
  series    = {Lecture Notes Series},
  volume    = {34},
  publisher = {Seoul National University, Research Institute of Mathematics, Global Analysis Research Center},
  address   = {Seoul, South Korea},
  year      = {1996},
  pages     = {iv+163},
  mrclass   = {93-02 (35B37 35Q30 93B05 93C20)},
  mrnumber  = {1406566},
  mrreviewer= {Vilmos Komornik}
}

@article {GC-HS-2021,
    AUTHOR = {Gonz\'{a}lez Casanova, Pedro and Hern\'{a}ndez-Santamar\'{\i}a, V\'{\i}ctor},
     TITLE = {Carleman estimates and controllability results for fully
              discrete approximations of 1{D} parabolic equations},
   JOURNAL = {Adv. Comput. Math.},
  FJOURNAL = {Advances in Computational Mathematics},
    VOLUME = {47},
      YEAR = {2021},
    NUMBER = {5},
     PAGES = {Paper No. 72, 71},
      ISSN = {1019-7168},
   MRCLASS = {93B05 (65M06)},
  MRNUMBER = {4314112},
MRREVIEWER = {Can Zhang},
       DOI = {10.1007/s10444-021-09885-4},
       URL = {},
}

@article {BHLR:2011,
    AUTHOR = {Boyer, F. and Hubert, F. and Le Rousseau, J.},
     TITLE = {Uniform controllability properties for space/time-discretized
              parabolic equations},
   JOURNAL = {Numer. Math.},
  FJOURNAL = {Numerische Mathematik},
    VOLUME = {118},
      YEAR = {2011},
    NUMBER = {4},
     PAGES = {601--661},
      ISSN = {0029-599X},
   MRCLASS = {93B40 (35K20 65M06 93B05 93B07)},
  MRNUMBER = {2822494},
MRREVIEWER = {Wenming Bian},
       DOI = {10.1007/s00211-011-0368-1},
       URL = {},
}

@article{Thuy:2014,
title = {Carleman estimates for semi-discrete parabolic operators with a discontinuous diffusion coefficient and applications to controllability},
journal = {Mathematical Control and Related Fields},
volume = {4},
number = {2},
pages = {203-259},
year = {2014},
issn = {2156-8472},
doi = {10.3934/mcrf.2014.4.203},
url = {},
author = {Thuy N. T. Nguyen}
}

@article {Thuy:2015,
    AUTHOR = {Nguyen, Thuy N. T.},
     TITLE = {Uniform controllability of semidiscrete approximations for
              parabolic systems in {B}anach spaces},
   JOURNAL = {Discrete Contin. Dyn. Syst. Ser. B},
  FJOURNAL = {Discrete and Continuous Dynamical Systems. Series B. A Journal
              Bridging Mathematics and Sciences},
    VOLUME = {20},
      YEAR = {2015},
    NUMBER = {2},
     PAGES = {613--640},
      ISSN = {1531-3492,1553-524X},
   MRCLASS = {93B05 (65J10 93C20)},
  MRNUMBER = {3331671},
MRREVIEWER = {Zhongcheng\ Zhou},
       DOI = {10.3934/dcdsb.2015.20.613},
       URL = {},
}

@incollection {B:2013,
    AUTHOR = {Boyer, F.},
     TITLE = {On the penalised {HUM} approach and its applications to the
              numerical approximation of null-controls for parabolic
              problems},
 BOOKTITLE = {C{ANUM} 2012, 41e {C}ongr\`es {N}ational d'{A}nalyse
              {N}um\'{e}rique},
    SERIES = {ESAIM Proc.},
    VOLUME = {41},
     PAGES = {15--58},
 PUBLISHER = {EDP Sci., Les Ulis},
      YEAR = {2013},
   MRCLASS = {49M25 (65J05)},
  MRNUMBER = {3174955},
MRREVIEWER = {Songting Luo},
       DOI = {10.1051/proc/201341002},
       URL = {},
}

@article{TZ-2009,
author = {Tang, Shanjian and Zhang, Xu},
title = {Null Controllability for Forward and Backward Stochastic Parabolic Equations},
journal = {SIAM Journal on Control and Optimization},
volume = {48},
number = {4},
pages = {2191-2216},
year = {2009},
doi = {10.1137/050641508},
URL = {https://doi.org/10.1137/050641508}}

@misc{ZXL:2025,
           title={New global {C}arleman estimates and null controllability for forward/backward semi-linear parabolic SPDEs}, 
      author={Lei Zhang and Fan Xu and Bin Liu},
      year={2025},
      eprint={2401.13455},
      archivePrefix={arXiv},
      primaryClass={math.OC},
      url={https://arxiv.org/abs/2401.13455}, 
}

@article {LPP:2025,
    AUTHOR = {Lecaros, Rodrigo and P\'erez, Ariel A. and Prado, Manuel F.},
     TITLE = {Carleman {E}stimate for {S}emi-discrete {S}tochastic
              {P}arabolic {O}perators in {A}rbitrary {D}imension and
              {A}pplications to {C}ontrollability},
   JOURNAL = {Appl. Math. Optim.},
  FJOURNAL = {Applied Mathematics and Optimization},
    VOLUME = {93},
      YEAR = {2026},
    NUMBER = {1},
     PAGES = {Paper No. 12},
      ISSN = {0095-4616,1432-0606},
   MRCLASS = {93B05 (65M06 93B07 93C20)},
  MRNUMBER = {5006586},
       DOI = {10.1007/s00245-025-10364-1},
       URL = {},
}
\end{document}